\documentclass[a4paper,11pt]{article}
\usepackage[margin=0.8in]{geometry}
\usepackage{enumitem}

\usepackage{braket}
\usepackage[backref=page,hidelinks]{hyperref}
\usepackage{amssymb,mathtools,amsthm,amsmath}
\usepackage{tikz}

\usepackage{xargs}

\usepackage{verbatim}

\theoremstyle{plain}
\newtheorem{theorem}{\bf Theorem}
\newtheorem{claim}[theorem]{\bf Claim}
\newtheorem{conjecture}[theorem]{\bf Conjecture}

\newtheorem{proposition}[theorem]{\bf Proposition}
\newtheorem{corollary}[theorem]{\bf Corollary}
\newtheorem{lemma}[theorem]{\bf Lemma}
\theoremstyle{definition}

\numberwithin{theorem}{section}
\numberwithin{equation}{section}

\newcommand{\Rea}{{\mathbb R}}

\newcommand{\ceilfrac}[2]{\left\lceil \frac{#1}{#2}\right\rceil}

\DeclareMathOperator{\eps}{\varepsilon}

\usepackage[d]{esvect}

\usepackage[nobysame,msc-links,non-sorted-cites,abbrev]{amsrefs}

\renewcommand{\eprint}[1]{\href{https://arxiv.org/abs/#1}{arXiv:#1}}

\BibSpec{book}{%
	+{}  {\PrintPrimary}                {transition}
	+{,} { \textbf}                     {title} 
	+{.} { }                            {part}
	+{:} { \textit}                     {subtitle}
	+{,} { \PrintEdition}               {edition}
	+{}  { \PrintEditorsB}              {editor}
	+{,} { \PrintTranslatorsC}          {translator}
	+{,} { \PrintContributions}         {contribution}
	+{,} { }                            {series}
	+{,} { \voltext}                    {volume}
	+{,} { }                            {publisher}
	+{,} { }                            {organization}
	+{,} { }                            {address}
	+{,} { \PrintDateB}                 {date}
	+{,} { }                            {status}
	+{}  { \parenthesize}               {language}
	+{}  { \PrintTranslation}           {translation}
	+{;} { \PrintReprint}               {reprint}
	+{.} { }                            {note}
	+{.} {}                             {transition}
	+{}  {\SentenceSpace \PrintReviews} {review}
}
\BibSpec{incollection}{%
  +{}  {\PrintAuthors}                {author}
  +{,} { \textit}                     {title}
  +{.} { }                            {part}
  +{:} { \textit}                     {subtitle}
  +{,} { \PrintContributions}         {contribution}
  +{,} { \PrintConference}            {conference}
  +{}  {\PrintBook}                   {book}
  +{,} { }                            {booktitle}
	+{}  { \PrintEditorsB}              {editor}
	+{,} { }                            {publisher}
  +{,} { \PrintDateB}                 {date}
  +{,} { pp.~}                        {pages}
  +{,} { }                            {status}
  +{,} { \PrintDOI}                   {doi}
  +{,} { available at \eprint}        {eprint}
  +{}  { \parenthesize}               {language}
  +{}  { \PrintTranslation}           {translation}
  +{;} { \PrintReprint}               {reprint}
  +{.} { }                            {note}
  +{.} {}                             {transition}
  +{}  {\SentenceSpace \PrintReviews} {review}
}

\makeatletter
\renewcommand{\PrintNames@a}[4]{%
	\PrintSeries{\name}
	{#1}
	{}{ and \set@othername}
	{,}{ \set@othername}
	{}{ and \set@othername}
	{#2}{#4}{#3}%
}
\makeatother

\usepackage{authblk}

\begin{document}

\title{Sums of Laplacian eigenvalues and sums of degrees}

    \author{Alan Lew\thanks{Dept. Math. Sciences, Carnegie Mellon University, Pittsburgh, PA 15213, USA; 
    Department of Mathematics, Technion, Haifa 32000, Israel. \quad e-mail: \href{mailto:alanlew@technion.ac.il}{alanlew@technion.ac.il}.}}
  
	\date{}
	\maketitle

\begin{abstract}
Let $X$ be a simplicial complex. For $1\le i\le\dim(X)$, let $X(i)$ be the set of $i$-dimensional faces of $X$, and let $f_i(X)=|X(i)|$. For $0\le i\le \dim(X)-1$, let $L_i^+(X)$ be the $i$-th upper Laplacian operator of $X$. For $\sigma\in X$ and $1\le r\le \dim(X)$, we denote by $\text{deg}_X^{(r)}(\sigma)$ the number of $r$-dimensional faces of $X$ containing $\sigma$. For a symmetric matrix $M\in \mathbb{R}^{n\times n}$ and $1\le i\le n$, let $\lambda_i(M)$ be the $i$-th largest eigenvalue of $M$. We prove that for every complex $X$, $1\le r\le\dim(X)$, and $1\le k\le f_{r-1}(X)/(r+1)$,
\[ \sum_{i=1}^k \lambda_i(L_{r-1}^+(X)) \le \max \left\{ \sum_{\sigma\in A} \text{deg}_X^{(r)}(\sigma) :\, A\subset X(r-1),\, |A|=(r+1)k \right\}.
\]
This bound is sharp, and it extends a classical result of Anderson and Morley, corresponding to the special case $k=1,\, r=1$. As a consequence, we show that for all $1\le r\le \dim(X)$ and $1\le k\le f_{r-1}(X)$, \[ \sum_{i=1}^{k} \lambda_i(L_{r-1}^+(X)) \le f_r(X) + \binom{(r+1)k}{2}. \] In the case $r=1$, we obtain the following improved bound: for every $k\ge 1$ and every graph $G=(V,E)$ with $|V|\ge k$,
\[
  \sum_{i=1}^k \lambda_i(L(G)) \leq |E|+k^2,
\]
where $L(G)=L_0^{+}(G)$ is the Laplacian matrix of $G$. This improves upon previously known bounds for all $k\ge 3$, and may be seen as a further step towards Brouwer's conjecture, which states that $\sum_{i=1}^k \lambda_i(L(G)) \leq |E|+\binom{k+1}{2}$.

As an additional application, we show that if $X$ is an $(r+1)$-partite $r$-dimensional simplicial complex on vertex set $V$, and $1\le k\le f_{r-1}(X)$, then 
\[
  \sum_{i=1}^{k} \lambda_i(L_{r-1}^+(X)) \le \sum_{i=1}^k \left|\{v\in V:\, \text{deg}^{(r)}_X(v)\ge i\}\right|.
\]
This resolves a special case of a conjecture of Duval and Reiner, which states that the above inequality holds for all simplicial complexes.
\end{abstract}


\section{Introduction}

Let $G=(V,E)$ be a finite, simple graph with $n$ vertices. For a vertex $v\in V$, we denote by $\deg_G(v)=|\{u\in V:\, \{u,v\}\in E\}|$ the \emph{degree} of $v$ in $G$. For $1\le i\le n$, let $d_i(G)$ be the $i$-th largest degree of a vertex in $G$. The \emph{Laplacian matrix} of $G$ is the matrix $L(G)\in \Rea^{n \times n}$ defined by
\[
    L(G)_{u,v}=\begin{cases}
        \deg_G(u) & \text{if } u=v,\\
        -1 & \text{if } \{u,v\}\in E,\\
        0 & \text{otherwise,}
    \end{cases}
\]
for all $u,v\in V$. 
For a symmetric matrix $M\in \Rea^{m\times m}$ and an index $1\le i\le m$, we denote by $\lambda_i(M)$ the $i$-th largest eigenvalue of $M$. It is well-known that $L(G)$ is positive semi-definite and has rank at most $n-1$ (see Section \ref{sec:prelims:incidence} for more details), and therefore its eigenvalues satisfy $\lambda_1(L(G))\ge \lambda_2(L(G))\ge \cdots\ge \lambda_{n-1}(L(G))\ge \lambda_n(L(G))=0$.

The Laplacian spectrum is tightly related to different combinatorial properties of the graph. For example, Kirchhoff's Matrix-Tree Theorem \cite{kirchhoff1847ueber} gives a formula for the number of spanning trees of a graph in terms of the product of its (non-trivial) Laplacian eigenvalues.
The second smallest Laplacian eigenvalue, also known as the \emph{algebraic connectivity} of the graph, is related to various connectivity and expansion properties of the graph (see, for example, \cite{fiedler1973algebraic,alon1985lambda1,mohar1989isoperimetric}) and to the behavior of certain random walks on the graph (see, for example, \cite{caputo2010aldous}).

The study of the relationship between the Laplacian eigenvalues of a graph and its degree sequence has its roots in the pioneering works of Anderson and Morley \cite{anderson1985eigenvalues} and Fiedler \cite{fiedler1973algebraic}. In particular, Anderson and Morley showed in \cite[Theorem 2]{anderson1985eigenvalues} that, for every graph $G=(V,E)$,
\begin{equation}\label{eq:anderson_morley}
    \lambda_1(L(G))\le d_1(G)+d_2(G).
\end{equation}
In fact, they proved the stronger inequality $\lambda_1(L(G))\le \max\{\deg_G(u)+\deg_G(v):\, \{u,v\}\in E\}$. Further results relating the Laplacian spectrum to the degrees of vertices in a graph appear, for example, in \cite{li1997upper,li2000note,brouwer2008lower}.

In \cite{grone1994laplacian}, Grone and Merris studied the relations between sums of Laplacian eigenvalues and sums of degrees in a graph. They observed that for every graph $G=(V,E)$ and $1\le k\le |V|$, 
\[
    \sum_{i=1}^k \lambda_i(L(G)) \ge \sum_{i=1}^k d_i(G).
\]
This bound was later improved by Grone in \cite{grone1995eigenvalues}. 
The \emph{conjugate degree sequence} of an $n$-vertex graph $G$ is the sequence $(d_1'(G),\ldots,d_{n}'(G))$ defined by
\[
    d'_i(G)= \left| \{ 1\le j\le n:\, d_j(G)\ge i\}\right|
\]
for all $1\le i\le n$. Grone and Merris conjectured in \cite{grone1994laplacian} that the sum of the $k$ largest Laplacian eigenvalues of a graph is bounded from above by the sum of the $k$ largest elements of its conjugate degree sequence. The case $k=1$ was verified in \cite{grone1994laplacian}, while the case $k=2$ was proved by Duval and Reiner in \cite{duval2002shifted}. The general case was proved by Bai in \cite{bai2011gronemerris}, building on earlier ideas by Katz \cite{katz2005grone}.

\begin{theorem}[Bai \cite{bai2011gronemerris}]\label{thm:bai}
Let $G=(V,E)$ be a graph. Then, for all $1\le k\le |V|$,
\[ 
    \sum_{i=1}^k \lambda_i(L(G))\le \sum_{i=1}^k d_i'(G).
\]
\end{theorem}

Motivated by Grone and Merris' conjecture, Brouwer conjectured the following.

\begin{conjecture}[Brouwer; see \cite{haemers2010onthesum,brouwer2012book}]\label{conj:brouwer}
    Let $k\ge 1$, and let $G=(V,E)$ be a graph with $|V|\ge k$. Then,
\[
    \sum_{i=1}^k \lambda_i(L(G))\le |E| + \binom{k+1}{2}.
\]
\end{conjecture}

Despite much interest in the problem in the last decade, Conjecture \ref{conj:brouwer} remains open. Some special cases are known to hold: The case $k=1$ follows easily, for example from Anderson and Morley's bound $\lambda_1(L(G))\le d_1(G)+d_2(G)\le |E|+1$. The case $k=2$ was proved by Haemers, Mohammadian, and Tayfeh-Rezaie in \cite{haemers2010onthesum}. The conjecture is also known to hold for some special classes of graphs, such as trees \cite{haemers2010onthesum}, split graphs \cite{berndsen2010thesis, mayard2010thesis}, and regular graphs \cite{berndsen2010thesis,mayard2010thesis}. A conjectural characterization of the graphs achieving equality in Brouwer's bound was proposed by Li and Guo in \cite{li2022full}. See, for example, \cite{chen2018note,chen2019onbrouwers, chen2025more, cooper2021constraints,du2012upper,fritscher2011onthesum, ganie2020further, guan2014onthesum, rocha2020aas, torres2024brouwer ,helmberg2017spectral} for more related work.

In \cite{lew2024partition}, we proved the following weak version of Conjecture \ref{conj:brouwer}.

\begin{theorem}[See {\cite[Theorem 1.3, Corollary 1.6]{lew2024partition}}]\label{thm:weak_brouwer}
    Let $k\ge 1$, and let $G=(V,E)$ be a graph with $|V|\ge k$. Then,
    \[
        \sum_{i=1}^k \lambda_i(L(G))\le |E| +\min \left\{2k^2 -\ceilfrac{k}{2}, k^2 + 15 k\log{k} +65k\right\},
    \]
    where $\log x$ denotes the natural logarithm of $x$.
\end{theorem}

In \cite{eckmann1945haromic}, Eckmann extended the notion of Laplacian operators from graphs to simplicial complexes. Recall that a simplicial complex $X$ on a finite vertex set $V$ is a family of subsets of $V$ that is closed under inclusion. That is, if $\tau\in X$ and $\sigma\subset \tau$, then $\sigma\in X$. The elements of $X$ are called the \emph{faces} or \emph{simplices} of $X$. The dimension of a simplex $\sigma\in X$ is $\dim(\sigma)=|\sigma|-1$. In particular, the empty set is considered a $(-1)$-dimensional simplex of $X$.
The dimension of the simplicial complex $X$, denoted by $\dim(X)$, is the maximal dimension of a simplex in $X$. For $-1\le i\le \dim(X)$, let $X(i)$ be the set of $i$-dimensional faces of $X$, and let $f_i(X)=|X(i)|$. We usually identify $X(0)$ with the vertex set $V$. We may identify a one-dimensional simplicial complex $X$ with the graph $G=(X(0),X(1))$.

Let $X$ be a simplicial complex on vertex set $V$. For $1\le r\le \dim(X)$ and $\sigma\in X$, the \emph{$r$-degree} of $\sigma$, denoted by $\deg_X^{(r)}(\sigma)$, is the number of $r$-dimensional simplices of $X$ containing $\sigma$. We fix an arbitrary linear order $<$ on the vertex set $V$. Let $0\le r\le \dim(X)$, $\sigma\in X(r-1)$, and $\tau\in X(r)$ such that $\sigma\subset \tau$. Let $u$ be the unique vertex in $\tau\setminus \sigma$. We define
\[
    (\tau:\sigma)= (-1)^{|\{v\in \tau:\, v<u\}|}.
\]
Let $1\le r\le \dim(X)$. The \emph{$(r-1)$-dimensional upper Laplacian matrix} of $X$ is the matrix $L_{r-1}^{+}(X)\in \Rea^{f_{r-1}(X)\times f_{r-1}(X)}$ defined by
\[
    L_{r-1}^{+}(X)_{\tau,\eta}= \begin{cases}
        \deg_{X}^{(r)}(\tau) & \text{if } \tau=\eta,\\
        -(\tau:\tau\cap \eta)(\eta:\tau\cap\eta) & \text{if } |\tau\cap \eta|=r-1,\, \tau\cup\eta\in X,\\
        0 & \text{otherwise,}
    \end{cases}
\]
for all $\tau,\eta\in X(r-1)$. Note that for a graph $G$, we have $L_0^{+}(G)=L(G)$.

Eckmann observed in \cite{eckmann1945haromic} that, in analogy with the Hodge theorem in Riemannian geometry, the $i$-dimensional cohomology (with real coefficients) of a simplicial complex can be determined by studying its $i$-dimensional Laplacian operator.
 Since their introduction, high-dimensional Laplacians have found numerous applications in areas such as combinatorics, topology, and algebra. For example, in his seminal work \cite{garland1973padic}, Garland studied the Laplacian spectra of Bruhat-Tits buildings associated with linear algebraic groups over non-archimedean local fields, and proved as a consequence a conjecture of Serre on the cohomology groups of finite quotients of these buildings. 
In recent years, new extensions and applications in diverse fields of Garland's arguments have been discovered (see, for example, \cite{ballmann1997l2,zuk1996propriete,kahle2014sharp,aharoni2005eigenvalues,oppenheim2018local,anari2024log}). For further work on high-dimensional Laplacian operators of simplicial complexes, see, for example, \cite{kalai1983enumeration,kook2000combinatorial,horak2013spectra}.

Duval and Reiner proposed in \cite{duval2002shifted} the following   generalization of the Grone-Merris conjecture.

\begin{conjecture}[Duval, Reiner {\cite[Conjecture 1.2]{duval2002shifted}}]\label{conj:duval_reiner}
Let $X$ be a simplicial complex on vertex set $V$, and let $1\le r\le \dim(X)$. Then, for all $1\le k\le f_{r-1}(X)$,
\[
    \sum_{i=1}^k \lambda_i(L_{r-1}^{+}(X)) \le \sum_{i=1}^k \left| \{ v\in V:\, \deg_X^{(r)}(v)\ge i\}\right|.
\]
\end{conjecture}
In \cite{abebepartial,abebe2019conjectural}, Abebe and Pfeffer considered various possible high-dimensional extensions of Brouwer's conjecture. Based on some numerical experiments, we propose the following variant, which is a somewhat stronger version of one of the conjectures in \cite{abebepartial} (see \cite[Section 3.2.3]{abebepartial}).

\begin{conjecture}\label{conj:higher_brouwer}
   Let $X$ be a simplicial complex on vertex set $V$, and let $1\le r\le \dim(X)$. Then, for all $1\le k\le f_{r-1}(X)$,
\[
    \sum_{i=1}^k \lambda_i(L_{r-1}^{+}(X)) \le f_{r}(X) + \binom{k}{2}+rk.
\] 
\end{conjecture}
See Section \ref{sec:conclusion:brower_for_complexes} for further details on this conjecture. 
In this paper, we continue the study of sums of Laplacian eigenvalues of graphs and simplicial complexes, and their relation to their degree sequences. Our main result is the following.

For a simplicial complex $X$, $1\le r\le \dim(X)$, and $1\le i\le f_{r-1}(X)$, let $d_i^{(r)}(X)$ be the $i$-th largest $r$-degree of a simplex in $X(r-1)$.

\begin{theorem}\label{thm:degree_bound}
    Let $X$ be a simplicial complex, $1\le r\le \dim(X)$, and $1\le k\le f_{r-1}(X)/(r+1)$. Then,
    \[
        \sum_{i=1}^k \lambda_i(L_{r-1}^{+}(X))\le \sum_{i=1}^{(r+1)k} d^{(r)}_i(X).
    \]
\end{theorem}

The $r=1,\, k=1$ case of Theorem \ref{thm:degree_bound} is exactly Anderson and Morley's bound \eqref{eq:anderson_morley}. The $r\ge 1, k=1$ case was recently proved by Fan, Wu, and Wang in \cite[Theorem 3.5]{fan2025largest} (in fact, they proved the stronger bound $\lambda_1(L_{r-1}^{+}(X))\le \max_{\tau\in X(r)} \left(\sum_{\sigma\in X(r-1),\sigma\subset \tau} \deg_X^{(r)}(\sigma)\right)$).

The bound in Theorem \ref{thm:degree_bound} is tight for all $r\ge 1$ and $k\ge 1$. Indeed, if $X$ is a simplicial complex in which every $(r-1)$-dimensional face is contained in exactly one $r$-dimensional simplex, then $\sum_{i=1}^k \lambda_i(L^+_{r-1}(X))=\sum_{i=1}^{(r+1)k}d_i^{(r)}(X)=(r+1)k$ for all $1\le k\le f_{r-1}(X)/(r+1)$. In the graphical case, the corresponding example is a perfect matching.  
It appears that no connected graph achieves equality in Theorem \ref{thm:degree_bound} for $k>1$. However, for every $k$, there exist connected graphs for which the difference between the sum of the $2k$ largest degrees and the sum of the $k$ largest Laplacian eigenvalues is arbitrarily close to zero (see Section \ref{sec:examples} for more details).

As a consequence of Theorem \ref{thm:degree_bound}, we obtain the following result.

\begin{corollary}\label{cor:weak_brouwer_complexes}Let $X$ be a simplicial complex, and let 
$1\le r\le \dim(X)$. Then, for all $1\le k\le f_{r-1}(X)$,
\[
\sum_{i=1}^{k} \lambda_i(L_{r-1}^{+}(X)) \le f_{r}(X) +  \binom{(r+1)k}{2}.
\]
\end{corollary}

Note that Corollary \ref{cor:weak_brouwer_complexes} is a weak version of the bound in Conjecture \ref{conj:higher_brouwer}. In the graphical case, we obtain, by combining Theorem \ref{thm:degree_bound} with Theorem \ref{thm:bai}, the following improved bound.

\begin{theorem}\label{thm:weak_brouwer_improved}
    Let $k\ge 1$, and let $G=(V,E)$ be a graph with $|V|\ge k$. Then,
    \[
        \sum_{i=1}^k \lambda_i(L(G))\le |E|+k^2.
    \]
\end{theorem}
Theorem \ref{thm:weak_brouwer_improved} improves upon the bounds in Theorem \ref{thm:weak_brouwer} for all $k\ge 2$, and it may be seen as a further step towards Conjecture \ref{conj:brouwer}.

As an additional application of Theorem \ref{thm:degree_bound}, we obtain, for various families of graphs, new upper bounds on the difference between the sum of the $k$ largest Laplacian eigenvalues and the number of edges of the graph. In particular, we show that the bound in Conjecture \ref{conj:brouwer} holds for square-free graphs for all $k\ge 7$, and for graphs with girth at least $5$ for all $k\ge 1$ (see Proposition \ref{prop:apps} for details). 

An $r$-dimensional simplicial complex $X$ is called \emph{$(r+1)$-partite} if there exists a partition $V_1,\ldots,V_{r+1}$ of its vertex set $V$, such that for every $\sigma\in X$, $|\sigma\cap V_i|\le 1$ for all $1\le i\le r+1$. We call $V_1,\ldots,V_{r+1}$ the \emph{partite sets} of $X$.  For $1\le j\le r+1$, we denote by $X(r-1;j)$ the set of $(r-1)$-dimensional simplices of $X$ that do not intersect $V_j$, and for $1\le i\le |X(r-1;j)|$, we denote by 
$d^{(r)}_i(X;j)$ the $i$-th largest $r$-degree of a simplex in $X(r-1;j)$. 
For $(r+1)$-partite $r$-dimensional complexes, the following stronger version of Theorem \ref{thm:degree_bound} holds.

\begin{theorem}\label{thm:degree_bound_partite}
    Let $r\ge 1$. Let $X$ be an $(r+1)$-partite $r$-dimensional simplicial complex, with partite sets $V_1,\ldots,V_{r+1}$.
    Let $1\le k\le f_{r-1}(X)$. Then,
    \[
        \sum_{i=1}^k \lambda_i(L^+_{r-1}(X))\le \sum_{j=1}^{r+1} \sum_{i=1}^{\min\{k,|X(r-1;j)|\}} d_{i}^{(r)}(X;j).
    \]
\end{theorem}

As a consequence, we show that the bounds in Conjecture \ref{conj:duval_reiner} hold for the $(r-1)$-dimensional upper Laplacian of $(r+1)$-partite $r$-dimensional complexes.

\begin{corollary}\label{cor:dr_partite}
      Let $X$ be an $(r+1)$-partite $r$-dimensional simplicial complex on vertex set $V$, and let $1\le k\le f_{r-1}(X)$. Then,
   \[
    \sum_{i=1}^{k} \lambda_i(L_{r-1}^{+}(X)) \le \sum_{i=1}^k \left| \{v\in V:\, \deg^{(r)}_X(v)\ge i\}\right|.
\]
\end{corollary}

The paper is organized as follows. In Section \ref{sec:prelims}, we present some preliminary results, which we will later use. In Section \ref{sec:main}, we present the proof of our main result, Theorem \ref{thm:degree_bound}, and its Corollary \ref{cor:weak_brouwer_complexes}. In addition, we present examples showing the tightness of the bound in Theorem \ref{thm:degree_bound}. Section \ref{sec:brouwer} contains the proof of Theorem \ref{thm:weak_brouwer_improved}. In Section \ref{sec:apps}, we present some additional applications of Theorem \ref{thm:degree_bound} in the graphical setting (in particular, we state and prove Proposition \ref{prop:apps} mentioned above). 
In Section \ref{sec:partite}, we prove Theorem \ref{thm:degree_bound_partite} and Corollary \ref{cor:dr_partite}, dealing with the $(r-1)$-th upper Laplacian spectrum of $(r+1)$-partite $r$-dimensional simplicial complexes. 
We conclude with Section \ref{sec:conclusion}, where we propose a stronger version of Conjecture \ref{conj:brouwer}, make some remarks on Conjecture \ref{conj:higher_brouwer}, and discuss the extent to which our results transfer to the setting of signless Laplacian matrices.

\section{Preliminaries}\label{sec:prelims}

\subsection{Eigenvalue bounds}

Recall that for a symmetric matrix $M\in\Rea^{n\times n}$ and $1\le i\le n$, we denote by $\lambda_i(M)$ the $i$-th largest eigenvalue of $M$.
We will need the following facts about the eigenvalues of symmetric matrices.

\begin{lemma}[Frobenius, Ger\v{s}gorin; see {\cite[I.1]{marshall2011inequalities}}, {\cite[Theorem 6.1.1]{horn2013matrix}}]
\label{lemma:frobenius}
    Let $M\in\Rea^{n\times n}$ be a symmetric matrix\footnote{In fact, the bound in Lemma \ref{lemma:frobenius} holds for the spectral radius of any complex square matrix, but we will not use here this more general fact.}. Then,
    \[
        \lambda_1(M)\le \max \left\{ \sum_{j=1}^n |M_{ij}| :\, 1\le i\le n\right\}
    \]
\end{lemma}

\begin{lemma}[Ky Fan; see {\cite[Proposition A.6]{marshall2011inequalities}}]\label{lemma:ky_fan}
    Let $A,B\in\Rea^{n\times n}$ be symmetric matrices, and let $1\le k\le n$. Then,
    \[
        \sum_{i=1}^k\lambda_i(A+B)\le \sum_{i=1}^k\lambda_i(A)+\sum_{i=1}^k \lambda_i(B).
    \]
\end{lemma}

Finally, we will need the following simple result on the Laplacian eigenvalues of disconnected graphs.

\begin{lemma}[see {\cite[Proposition 1.3.6]{brouwer2012book}}]\label{lemma:components}
    Let $G$ be a graph with connected components $G_1,\ldots, G_s$. Then, the spectrum of $L(G)$ is the union of the spectra of $L(G_i)$, for $1\le i\le s$ (where the multiplicities of the eigenvalues are added).
\end{lemma}

\subsection{Boundary matrices, upper and lower Laplacians}\label{sec:prelims:incidence}

Let $X$ be a simplicial complex on vertex set $V$, and let $<$ be a linear order on $V$. Let $0\le r\le \dim(X)$. Recall that for $\tau\in X(r)$ and $\sigma\in X(r-1)$ with $\sigma\subset \tau$, we defined $(\tau:\sigma)=(-1)^{|\{v\in\tau: v<u\}|}$, where $u$ is the unique vertex in $\tau\setminus \sigma$. We will need the following simple lemma, which appears implicitly, for example, in \cite{duval2002shifted}. For completeness, we include a proof.
\begin{lemma}\label{lemma:signs}
    Let $X$ be a simplicial complex, and let $1\le r\le \dim(X)$. Let $\tau,\eta\in X(r-1)$ such that $|\tau\cap \eta|=r-1$ and $\tau\cup\eta\in X$. Then,
    \[
        (\tau\cup\eta:\tau)(\tau\cup\eta:\eta)=-(\tau:\tau\cap \eta)(\eta:\tau\cap \eta).
    \]
\end{lemma}
\begin{proof}
    Let $u$ be the unique vertex in $\tau\setminus\eta$,  and let $w$ be the unique vertex in $\eta\setminus \tau$. Assume without loss of generality that $u<w$. Then,
    \[
      (\tau\cup\eta:\tau)(\tau\cup\eta:\eta)= (-1)^{|\{v\in \tau\cup \eta:\, v<w\}|}   \cdot (-1)^{|\{v\in \tau\cup \eta:\, v<u\}|} = (-1)^{|\{v\in\tau\cup\eta:\, u\le v<w\}|}= -(-1)^{|v\in \tau\cap\eta:\, u<v<w|\}}.
    \]
    On the other hand,
    \[
    (\tau:\tau\cap\eta)(\eta:\tau\cap\eta)= (-1)^{|\{v\in \tau:\, v<u\}|}   \cdot (-1)^{|\{v\in \eta:\, v<w\}|} = (-1)^{|v\in \tau\cap\eta:\, u<v<w|\}}.
    \]
    So $(\tau\cup\eta:\tau)(\tau\cup\eta:\eta)=-(\tau:\tau\cap \eta)(\eta:\tau\cap \eta)$, as wanted.
\end{proof}

The \emph{$r$-dimensional boundary operator} $B_{r}(X)\in \Rea^{f_{r-1}(X)\times f_r(X)}$ is defined by
\[
B_r(X)_{\sigma,\tau}= \begin{cases}
                    (\tau:\sigma) & \text{if } \sigma\subset \tau,\\
                    0 & \text{otherwise,}
\end{cases}     
\]
for all $\sigma\in X(r-1)$ and $\tau\in X(r)$. It is easy to verify, using Lemma \ref{lemma:signs}, that $L_{r-1}^{+}(X)=B_r(X) B_r(X)^{T}$. Therefore, $L_{r-1}^{+}(X)$ is positive semi-definite. Moreover, it is well-known and not hard to check that $B_{r-1}(X) B_{r}(X)=0$, and therefore the image of $B_{r-1}(X)^T$ is contained in $\ker L_{r-1}^{+}(X)$. 
Note that, in the one-dimensional case, $B_1(G)^T$ is the incidence matrix of $G$ (associated with the orientation induced by the order $<$), and  $B_{0}(G)^{T}\in \Rea^{|V|\times 1}$ is the all-ones vector. Therefore,  $\lambda_{|V|}(L(G))=0$ (see, for example, \cite{fiedler1973algebraic}). Throughout the paper, we will occasionally use the simple fact that $\sum_{i=1}^{|V|-1}\lambda_i(L(G))=\sum_{i=1}^{|V|}\lambda_i(L(G))=\text{Trace}(L(G))=2|E|$. 

Let $X$ be a simplicial complex, and let $0\le r\le \dim(X)$. 
The \emph{$r$-dimensional lower Laplacian} of $X$ is the linear operator $L^{-}_r(X)\in \Rea^{f_r(X)\times f_r(X)}$ defined by
\[
    L_r^{-}(X)= B_r(X)^T B_r(X).
\]
It is easy to check that
\begin{equation}\label{eq:Lminus_formula}
    L_r^{-}(X)_{\tau,\eta}=\begin{cases}
        r+1 & \text{if } \tau=\eta,\\
        (\tau:\tau\cap \eta)(\eta:\tau\cap \eta) & \text{if } |\tau\cap \eta|=r,\\
        0 & \text{otherwise,}
    \end{cases}     
\end{equation}
for all $\tau,\eta\in X(r)$. The $1$-dimensional lower Laplacian of a graph $G$, $L_{1}^{-}(G)$, is sometimes called the \emph{edge Laplacian} of $G$ (see, for example, \cite{merris1994survey} and the references therein).

We will need the following well-know fact.

\begin{lemma}[see {\cite[Chapter 9, A.1.a]{marshall2011inequalities}}]\label{lemma:AAT}
    Let $A\in \Rea^{n\times m}$ and $B\in \Rea^{m\times n}$. Then, the non-zero eigenvalues of $AB$ are the same as those of $BA$ (with the same multiplicities).
\end{lemma}

As an immediate consequence, we obtain the following.

\begin{corollary}[see, for example, \cite{duval2002shifted}]\label{cor:Lminus_L}
    Let $X$ be a simplicial complex, and let $1\le r\le \dim(X)$. Then, the non-zero eigenvalues of $L_{r-1}^{+}(X)$ are the same as those of $L^-_r(X)$ (with the same multiplicities).
\end{corollary}

\subsection{Laplacian eigenvalues of graph complements}

For a graph $G=(V,E)$, let $\overline{G}=(V,\overline{E})$, where $\overline{E}= \left\{e\subset V:\, |e|=2,\,  e\notin E\right\}$,  be the \emph{complement} of $G$. We will need the following basic fact about the Laplacian eigenvalues of the complement of a graph.

\begin{lemma}[Anderson and Morley \cite{anderson1985eigenvalues}, Kelmans; see {\cite[Theorem 3.3]{hammer1996laplacian}}]\label{lemma:complement}
    Let $G=(V,E)$ be a graph with $|V|=n$, and let $1\le i\le n-1$. Then,
    \[
        \lambda_i(L(G))= n- \lambda_{n-i}(L(\overline{G})).
    \]
\end{lemma}

For a graph $G=(V,E)$ and $1\le k\le |V|$, we denote
\[
    \eps_k(G)= \left(\sum_{i=1}^k \lambda_i(L(G))\right)-|E|.
\]
As a simple consequence of Lemma \ref{lemma:complement}, we obtain the following result.

\begin{lemma}[see {\cite[Theorem 3.1]{chen2019onbrouwers}}]
\label{lemma:complement_eps}
    Let $G=(V,E)$ be a graph with $|V|=n$. Then, for all $1\le k\le n-1$,
    \[
        \eps_k(G)= \eps_{n-k-1}(\overline{G}) +nk -\binom{n}{2}.
    \]
\end{lemma}
\begin{proof}
    Let $1\le k\le n-1$. Denote the edge set of $\overline{G}$ by $\overline{E}$. By Lemma \ref{lemma:complement}, we have
    \begin{align*}
        \sum_{i=1}^{k} \lambda_i(L(G)) &= \sum_{i=1}^k  (n-\lambda_{n-i}(L(\overline{G})))
        =nk - \sum_{i=1}^{n-1} \lambda_i (L(\overline{G})) + \sum_{i=1}^{n-k-1} \lambda_i(L(\overline{G}))
        \\
        &= nk - 2|\overline{E}| + \eps_{n-k-1}(\overline{G})+|\overline{E}|= nk+ |E|-\binom{n}{2} + \eps_{n-k-1}(\overline{G}).
    \end{align*}
    So, $\eps_k(G)=\eps_{n-k-1}(\overline{G}) +nk -\binom{n}{2}$, as wanted.
\end{proof}

\section{Proof of Theorem \ref{thm:degree_bound} and Corollary \ref{cor:weak_brouwer_complexes}}\label{sec:main}

In this section, we prove our main result, Theorem \ref{thm:degree_bound}, and its Corollary \ref{cor:weak_brouwer_complexes}. 
We will need the following lemma, which we will also use later, in Section \ref{sec:partite}.

\begin{lemma}\label{lemma:LAminus}
    Let $X$ be a simplicial complex, and let $1\le r\le \dim(X)$. Let $A\subset X(r-1)$ such that every $r$-dimensional face of $X$ contains at most one element of $A$. Let $L_A\in \Rea^{f_r(X)\times f_r(X)}$ be defined by
    \[
    (L_A)_{\tau,\eta}=\begin{cases}
        1 & \text{if } \tau=\eta,\, \sigma\subset \tau \text{ for some } \sigma\in A,\\
        (\tau:\tau\cap \eta)(\eta:\tau\cap \eta) & \text{if } |\tau\cap \eta|=r,\, \tau\cap \eta\in A,\\
        0 & \text{otherwise,}
    \end{cases}
    \]
    for all $\tau,\eta\in X(r)$. Then, the multi-set of non-zero eigenvalues of $L_A$ is exactly
    \[
\left\{ \deg_X^{(r)}(\sigma):\, \sigma\in A \text{ such that } \deg_X^{(r)}(\sigma)>0\right\}.
    \]
\end{lemma}
\begin{proof}
    Let $B\in \Rea^{|A|\times f_r(X)}$ be defined by
    \[
        B_{\sigma,\tau}=\begin{cases}
            (\tau:\sigma) & \text{if } \sigma\subset \tau,\\
            0 & \text{otherwise,}
        \end{cases}
    \]
    for all $\sigma\in A$ and $\tau\in X(r)$.
    It is easy to check, using the fact that every simplex in $X(r)$ contains at most one element of $A$, that
    \[
        L_{A}= B^T B.
    \]
    On the other hand, it is easy to check (again using the fact that every $\tau\in X(r)$ contains at most one simplex in $A$) that $B B^T\in \Rea^{|A|\times |A|}$ is a diagonal matrix with diagonal elements $\deg_X^{(r)}(\sigma)$, for $\sigma\in A$. In particular, the eigenvalues of $B B^{T}$ are exactly $\{\deg_X^{(r)}(\sigma):\, \sigma\in A\}$. Therefore, the claim follows from Lemma \ref{lemma:AAT}.
\end{proof}

\begin{proof}[Proof of Theorem \ref{thm:degree_bound}]

Let $X$ be a simplicial complex. Let $1\le r\le \dim(X)$, and let $1\le k\le f_{r-1}(X)/(r+1)$.
First, assume that $f_r(X)<k$. By Corollary \ref{cor:Lminus_L}, $L_{r-1}^+(X)$ has at most $f_r(X)<k$ non-zero eigenvalues, and therefore
\[
\sum_{i=1}^k \lambda_i(L^+_{r-1}(X))= \sum_{i=1}^{f_{r-1}(X)}\lambda_i(L_{r-1}^+(X)) =\text{Trace}(L_{r-1}^{+}(X))= \sum_{\sigma\in X(r-1)} \deg_X^{(r)}(\sigma).
\]
On the other hand, since every $\tau\in X(r)$ contains $r+1$ $(r-1)$-dimensional faces, the number of $(r-1)$-dimensional simplices in $X$ with non-zero $r$-degree is at most $(r+1)f_r(X)<(r+1)k$. Therefore, $\sum_{\sigma\in X(r-1)} \deg_X^{(r)}(\sigma)= \sum_{i=1}^{(r+1)k} d_i^{(r)}(X)$, as wanted. 

So, we may assume $k\le f_r(X)$. 
We order the $(r-1)$-dimensional faces of $X$ as $\sigma_1,\ldots,\sigma_{f_{r-1}(X)}$, where $\deg_X^{(r)}(\sigma_i)=d^{(r)}_i(X)$ for all $1\le i\le f_{r-1}(X)$. 
For $1\le i \le (r+1)k$, let $L_i\in \Rea^{f_r(X)\times f_r(X)}$ be defined by
\begin{equation}\label{eq:L_i_def}
    (L_i)_{\tau,\eta}=\begin{cases}
        1 & \text{if } \tau=\eta,\, \sigma_i\subset \eta,\\
        (\tau:\tau\cap\eta)(\eta:\tau\cap \eta) & \text{if } \tau\cap \eta=\sigma_i,
        \\  0 &\text{otherwise,}
    \end{cases}     
\end{equation}
for all $\tau,\eta\in X(r)$. 

\begin{claim}\label{claim:Li}
    Let $1\le i\le (r+1)k$. Then,
    \[
        \sum_{j=1}^k \lambda_j(L_i) = d_i^{(r)}(X).
    \]
\end{claim}
\begin{proof}
By Lemma \ref{lemma:LAminus}, $d^{(r)}_i(X)$ is an eigenvalue of $L_i$ with multiplicity one, and all other eigenvalues are equal to $0$. Therefore, $  \sum_{j=1}^k \lambda_j(L_i) = d_i^{(r)}(X)$.
\end{proof}

Let $d= d^{(r)}_{(r+1)k}(X)$, and let
\begin{equation}\label{eq:L'_def}
    L' = L_r^{-}(X)- \sum_{i=1}^{(r+1)k} \left(1-\frac{d}{d^{(r)}_i(X)}\right)L_i.
\end{equation}
For $\sigma\in X(r-1)$, let $w(\sigma)=\min\{d/\deg^{(r)}_X(\sigma),1\}$. Note that $w(\sigma_i)=d/d_i^{(r)}(X)$ for $1\le i\le (r+1)k$, and $w(\sigma_i)=1$ for $(r+1)k\le i\le f_{r-1}(X)$. 
We will need the following facts about the matrix~$L'$.
\begin{claim}\label{claim:L'1}
    For $\tau,\eta\in X(r)$, we have
    \[
    L'_{\tau,\eta}=\begin{dcases}
    \sum_{\substack{\sigma\in X(r-1),\\ \sigma\subset \tau}} w(\sigma) & \text{if } \tau=\eta,\\
    w(\tau\cap \eta) (\tau:\tau\cap \eta)(\eta:\tau\cap \eta) & \text{if } |\tau\cap \eta|=r,\\
    0 & \text{otherwise.}
    \end{dcases}
    \]
\end{claim}
\begin{proof}
First, note that, for $\tau\in X(r)$, we have, by \eqref{eq:L'_def}, \eqref{eq:Lminus_formula}, and \eqref{eq:L_i_def}, 
\begin{align*}
   L'_{\tau,\tau} &=
       r+1 - \sum_{\substack{i\in\{1,\ldots,(r+1)k\},\\\sigma_i\subset \tau}}\left(1-\frac{d}{d_i^{(r)}(X)}\right)
      \\& = \sum_{\substack{\sigma\in X(r-1),\\ \sigma\subset \tau}} 1 -\sum_{\substack{i\in\{1,\ldots,(r+1)k\},\\\sigma_i\subset \tau}}\left(1-\frac{d}{d_i^{(r)}(X)}\right) = \sum_{\substack{\sigma\in X(r-1),\\ \sigma\subset \tau}} w(\sigma).
\end{align*}
Now, let $\tau,\eta\in X(r)$ such that $\tau\ne \eta$. If $|\tau\cap \eta|<r$,  we have, by \eqref{eq:L'_def}, \eqref{eq:Lminus_formula}, and \eqref{eq:L_i_def},  $L'_{\tau,\eta}=0$. If $|\tau\cap \eta|=r$, we divide into two cases. If $\tau\cap \eta\notin \{\sigma_1,\ldots,\sigma_{(r+1)k}\}$, then
\[
L'_{\tau,\eta}= (\tau:\tau\cap \eta)(\eta:\tau\cap \eta)= w(\tau\cap \eta) (\tau:\tau\cap \eta)(\eta:\tau\cap \eta).
\]
Otherwise, $\tau\cap \eta=\sigma_i$ for some $1\le i\le (r+1)k$. Then,
\[
L'_{\tau,\eta}= (\tau:\tau\cap \eta)(\eta:\tau\cap \eta)-\left(1-\frac{d}{d_i^{(r)}(X)}\right)(\tau:\tau\cap \eta)(\eta:\tau\cap \eta) =w(\tau\cap \eta) (\tau:\tau\cap \eta)(\eta:\tau\cap \eta).
\]
\end{proof}

\begin{claim}\label{claim:L'2}
$
    \lambda_1(L')\le (r+1)d.
 $
\end{claim}
\begin{proof}
    Let $\tau\in X(r)$. By Claim \ref{claim:L'1},
    \begin{align*}
    \sum_{\eta\in X(r)} |L'_{\tau,\eta}| &= \sum_{\substack{\sigma\in X(r-1),\\ \sigma\subset \tau}} w(\sigma) + \sum_{\substack{\eta\in X(r),\\ |\tau\cap\eta|=r}} w(\tau\cap \eta)
   = \sum_{\substack{\sigma\in X(r-1),\\ \sigma\subset \tau}} w(\sigma) +\sum_{\substack{\sigma\in X(r-1),\\ \sigma\subset \tau}} \sum_{\substack{\eta\in X(r),\\ \tau\cap\eta=\sigma}} w(\sigma)
\\ & = \sum_{\substack{\sigma\in X(r-1),\\ \sigma\subset \tau}} w(\sigma) +\sum_{\substack{\sigma\in X(r-1),\\ \sigma\subset \tau}} \left(\deg_X^{(r)}(\sigma)-1\right)\cdot w(\sigma) =\sum_{\substack{\sigma\in X(r-1),\\ \sigma\subset \tau}} \deg_X^{(r)}(\sigma)\cdot w(\sigma)
\\ & \le \sum_{\substack{\sigma\in X(r-1),\\ \sigma\subset \tau}} d = (r+1)d,
    \end{align*}
where the last inequality follows from the definition of $w(\sigma)$. By Lemma \ref{lemma:frobenius}, 
\[
\lambda_1(L')\le \max_{\tau\in X(r)}\left(\sum_{\eta\in X(r)} |L'_{\tau,\eta}|\right)\le  (r+1)d.
\]
\end{proof}

Finally, by Corollary \ref{cor:Lminus_L}, Lemma \ref{lemma:ky_fan}, Claim \ref{claim:L'2} and Claim \ref{claim:Li}, we have
\begin{align*}
   \sum_{i=1}^k \lambda_i(L^+_{r-1}(X)) &= \sum_{i=1}^k \lambda_i(L^-_{r}(X)) \le \sum_{i=1}^k \lambda_i(L') + \sum_{i=1}^{(r+1)k}
    \left(1-\frac{d}{d_i^{(r)}(X)}\right)\sum_{j=1}^k\lambda_j(L_i) 
    \\ &\le k\cdot (r+1)d + \sum_{i=1}^{(r+1)k}\left(d_i^{(r)}(X)-d\right)
     = \sum_{i=1}^{(r+1)k}d_i^{(r)}(X).
\end{align*}

\end{proof}

\begin{proof}[Proof of Corollary \ref{cor:weak_brouwer_complexes}]

We enumerate the $(r-1)$-dimensional simplices of $X$ as $\sigma_1,\ldots,\sigma_{f_{r-1}(X)}$, where $\deg_X^{(r)}(\sigma_i)=d_i^{(r)}(X)$ for all $1\le i\le f_{r-1}(X)$. Let $N=\min\{(r+1)k,f_{r-1}(X)\}$, and let $A=\{\sigma_1,\ldots,\sigma_{N}\}$.
If $(r+1)k\le f_{r-1}(X)$, then, by Theorem \ref{thm:degree_bound},
\[
    \sum_{i=1}^k \lambda_i(L^+_{r-1}(X)) \le \sum_{i=1}^{(r+1)k} \deg_X^{(r)}(\sigma_i)=\left|\left\{ (\sigma,\tau) : \, \sigma\in A,\, \tau\in X(r),\, \sigma\subset \tau\right\}\right|.
\]
If $(r+1)k>f_{r-1}(X)$, then
\[
 \sum_{i=1}^{k} \lambda_i(L^+_{r-1}(X)) 
 \le  \sum_{i=1}^{f_{r-1}(X)} \lambda_i(L^+_{r-1}(X)) 
 =\sum_{i=1}^{f_{r-1}(X)} \deg_X^{(r)}(\sigma_i)=\left|\left\{ (\sigma,\tau) : \, \sigma\in A,\, \tau\in X(r),\, \sigma\subset \tau\right\}\right|.
\]
So, in both cases,
\[
  \sum_{i=1}^k \lambda_i(L^+_{r-1}(X)) \le \left|\left\{ (\sigma,\tau) : \, \sigma\in A,\, \tau\in X(r),\, \sigma\subset \tau\right\}\right|.
\]
For $\tau\in X(r)$, let $\partial\tau=\{\sigma\in X(r-1):\, \sigma\subset\tau\}$. 
For $0\le i\le r+1$, let $c_i=\left|\left\{\tau\in X(r):\, |\partial\tau\cap A|=i\right\}\right|$. We have
\[
\left|\left\{ (\sigma,\tau) : \, \sigma\in A,\, \tau\in X(r),\, \sigma\subset \tau\right\}\right| = \sum_{i=1}^{r+1} i\cdot c_i \le f_{r}(X) + \sum_{i=2}^{r+1} (i-1)c_i. 
\]
Let $C=\{\{\sigma,\sigma'\}:\, \sigma,\sigma'\in A,\, |\sigma\cap \sigma'|=r-1\}$. Note that $|C|\le \binom{|A|}{2}\le \binom{(r+1)k}{2}$, and that if $\{\sigma,\sigma'\}\in C$, then there is at most one $\tau\in X(r)$ containing both $\sigma$ and $\sigma'$. On the other hand, for $\tau\in X(r)$, if $|\partial \tau\cap A|=i$, then there are exactly $\binom{i}{2}$ pairs $\{\sigma,\sigma'\}\in C$ such that $\sigma\subset\tau$ and $\sigma'\subset\tau$. Since $i-1\le \binom{i}{2}$ for all $i\ge 2$, we obtain
\[
 \sum_{i=1}^k \lambda_i(L^+_{r-1}(X)) \le f_{r}(X) + \sum_{i=2}^{r+1} (i-1)c_i \le f_r(X)+ \sum_{i=2}^{r+1} \binom{i}{2}c_i \le f_r(X)+ \binom{(r+1)k}{2}.
\]
\end{proof}

\subsection{Extremal examples}\label{sec:examples}

The bound in Theorem \ref{thm:degree_bound} is tight, as shown by the following example.

\begin{proposition}\label{prop:extremal}
    Let $r\ge 1$.  Let $X$ be a simplicial complex in which every $(r-1)$-dimensional simplex is contained in exactly one $r$-dimensional simplex. Then, for all $1\le k\le f_{r-1}(X)/(r+1)$,
    \[
        \sum_{i=1}^k \lambda_i(L_{r-1}^+(X)) = (r+1)k = \sum_{i=1}^{(r+1)k} d_i^{(r)}(X).
    \]
\end{proposition}
\begin{proof}
    Let $X$ be a simplicial complex in which every $(r-1)$-dimensional simplex is contained in exactly one $r$-dimensional simplex. Note that $f_{r}(X) = f_{r-1}(X)/(r+1)$. 
It is easy to check, by \eqref{eq:Lminus_formula}, that $L^-_r(X)$ is the scalar matrix $(r+1)I$. 
  Hence, by Corollary \ref{cor:Lminus_L}, the unique non-zero eigenvalue of $L_{r-1}^+(X)$ is $r+1$, and its multiplicity is $f_r(X)=f_{r-1}(X)/(r+1)$. As a consequence, for all $1\le k\le f_{r-1}(X)/(r+1)$,
    \[
        \sum_{i=1}^k \lambda_i(L_{r-1}^+(X)) = (r+1)k.
    \]
    On the other hand, since every $(r-1)$-dimensional simplex has $r$-degree exactly one, we have
    \[
        \sum_{i=1}^{(r+1)k} d_i^{(r)}(X) = (r+1)k.
    \]
\end{proof}

It is easy to construct examples of simplicial complexes satisfying the property in Proposition \ref{prop:extremal}. For example, let $\sigma,\tau_1,\ldots,\tau_m$ be pairwise disjoint sets, with $|\sigma|\le r-1$ and $|\tau_i|=r+1-|\sigma|$ for all $1\le i\le m$.  Let $X$ be the simplicial complex whose maximal faces are $\sigma\cup\tau_i$, for $1\le i\le m$. 
For $r=1$, $X$ is just the graph consisting of a perfect matching with $m$ edges. It is clear that $X$ satisfies the property in Proposition \ref{prop:extremal}.

Next, we present some additional extremal examples in the case $r=1$, that is, the graphical case.
We denote by $S_n$ the star graph with $n$ vertices (that is, the graph with vertex set $[n]=\{1,2,\ldots,n\}$ and edge set $\{\{1,i\}: 2\le i\le n\}$).

\begin{proposition}
    Let $1\le k \le n/2$, and let $n_1,\ldots,n_k\ge 2$ such that $n_1+\cdots+n_k=n$. 
    Let $G=(V,E)$ be a forest on $n$ vertices with $k$ connected components, isomorphic to the star graphs $S_{n_1},\ldots,S_{n_k}$, respectively. Then,
    \[
        \sum_{i=1}^k \lambda_i(L(G))= \sum_{i=1}^{2k} d_i(G).
    \]
\end{proposition}
\begin{proof}
    Let $1\le i\le k$. It is well-known and easy to check (see, for example, \cite[Lemma 2]{haemers2010onthesum}) that the Laplacian eigenvalues of $L(S_{n_i})$ are $n_i$, with multiplicity $1$, $1$ with multiplicity $n_i-2$, and $0$ with multiplicity $1$. So, by Lemma \ref{lemma:components}, the $k$ largest Laplacian eigenvalues of $G$ are $n_1,\ldots,n_k$. On the other hand, the sum of the $2k$ largest degrees in $G$ is
    \[
        \sum_{i=1}^{2k}d_i(G)= (n_1-1)+\cdots+(n_k-1)+k\cdot 1= \sum_{i=1}^k n_i = \sum_{i=1}^k \lambda_i(L(G)).
    \]
\end{proof}

We could not find examples of connected graphs achieving equality in Theorem \ref{thm:degree_bound} for $k>1$, and we suspect that such graphs do not exist. However, for every $\epsilon>0$, there exist connected graphs for which the difference between the sum of the $2k$ largest degrees and the sum of $k$ largest Laplacian eigenvalues is at most $\epsilon$, as shown next. 
For $n\ge 1$, we denote by $P_n$ the path graph with $n$ vertices.

\begin{proposition}
    Let $k_0\ge 1$ and $\epsilon>0$. Then, there exists $n_0=n_0(k_0,\epsilon)$ such that, for all $n\ge n_0$ and $1\le k\le k_0$, the graph $P_n$ satisfies
    \[
       \sum_{i=1}^{2k}d_i(P_n)- \sum_{i=1}^k \lambda_i(L(P_n))<\epsilon.
    \]
\end{proposition}
\begin{proof}
   Let $n\ge 2k_0+2$, and let $1\le k\le k_0$. Note that $\sum_{i=1}^{2k} d_i(P_n)=2k\cdot 2= 4k$. 
   On the other hand, it is well-known (see, for example, \cite[p. 9]{brouwer2012book}) that the eigenvalues of $L(P_n)$ are exactly
   \[
     2-2\cos(\pi j/n),
   \]
   for $j=0,\ldots,n-1$. Therefore,
   \[
    \sum_{i=1}^k \lambda_i(L(P_n)) = 2k-2\sum_{i=1}^k \cos(\pi (n-i)/n).
   \]
   For fixed $i$, $\cos(\pi (n-i)/n)$ tends to $-1$ from above as $n$ goes to infinity. Therefore, there exists $n_0=n_0(k_0,\epsilon)\ge 2k_0+2$ such that, for all $1\le i\le k_0$ and $n\ge n_0$, $\cos(\pi (n-i)/n)<-1+\epsilon/(2k_0)$. So, for all $n\ge n_0$ and $1\le k\le k_0$, we obtain
   \[
    \sum_{i=1}^k \lambda_i(L(P_n))> 2k - 2k(-1+\epsilon/(2k_0))= 4k-\epsilon\cdot k/k_0\ge 4k-\epsilon =\left(\sum_{i=1}^{2k}d_i(P_n)\right)-\epsilon,
   \]
   as wanted.
\end{proof}

\section{Proof of Theorem \ref{thm:weak_brouwer_improved}}\label{sec:brouwer}

Recall that for a graph $G=(V,E)$ and $1\le k\le |V|$, we denote $\eps_k(G)=\left(\sum_{i=1}^k \lambda_i(L(G))\right)-|E|$. 
By combining our main result, Theorem \ref{thm:degree_bound}, with Bai's theorem (Theorem \ref{thm:bai}), we obtain the following bound. 

\begin{theorem}\label{thm:main+bai}
    Let $G=(V,E)$ be a graph with $|V|=n$, and let $1\le k\le (n-1)/2$. Then,
    \[
        \sum_{i=1}^k \lambda_i(L(G)) \le |E|+\frac{1}{2}\left( \sum_{i=1}^{2k} \min\{d_i(G),k\} - \sum_{i=2k+1}^n \max\{0,d_i(G)-k\}\right).
    \]
\end{theorem}
\begin{proof}
    By Theorem \ref{thm:degree_bound}, and using the fact that $\sum_{i=1}^n d_i(G)=2|E|$, we obtain
\begin{equation}\label{eq:1}
    \sum_{i=1}^k \lambda_i(L(G)) \le \sum_{i=1}^{2k} d_i(G) = |E| +\frac{1}{2}\sum_{i=1}^{2k}d_i(G) -\frac{1}{2}\sum_{i=2k+1}^{n} d_i(G).
\end{equation}
For $1\le i\le n$, we have $d_i(G)=\min\{d_i(G),k\}+\max\{0,d_i(G)-k\}$. Thus, we can write \eqref{eq:1} as
\begin{align}\label{eq:1'}
    2\eps_k(G) &\le \sum_{i=1}^{2k}\min\{d_i(G),k\} +  \sum_{i=1}^{2k}\max\{0,d_i(G)-k\}  \nonumber
    \\&-\sum_{i=2k+1}^n\min \{d_i(G),k\} - \sum_{i=2k+1}^n \min\{0,d_i(G)-k\}.
\end{align}

On the other hand, by Theorem \ref{thm:bai}, and using the fact that $\sum_{i=1}^n d_i'(G)=2|E|$, we have
\begin{equation}\label{eq:2}
    \sum_{i=1}^k \lambda_i(L(G))\le \sum_{i=1}^k d_i'(G) = |E|+ \frac{1}{2} \sum_{i=1}^k d_i'(G) - \frac{1}{2}\sum_{i=k+1}^n d_i'(G).
\end{equation}
Note that
\[
    \sum_{i=1}^k d_i'(G) = \sum_{i=1}^k |\{1\le j\le n:\, d_j(G)\ge i\}| = \sum_{i=1}^n \min\{d_i(G),k\}.
\]
Similarly,
\[
    \sum_{i=k+1}^n d_i'(G) = \sum_{i=k+1}^n |\{1\le j\le n:\, d_j(G)\ge i\}| = \sum_{i=1}^n \max\{0,d_i(G)-k\}.
\]
Therefore, we may write \eqref{eq:2} as
\begin{equation}\label{eq:2'}
    2\eps_k(G)\le  \sum_{i=1}^n \min\{d_i(G),k\}- \sum_{i=1}^n \max\{0,d_i(G)-k\}.
\end{equation}
Adding \eqref{eq:1'} to \eqref{eq:2'}, we obtain
\[
    4 \eps_k(G)\le 2 \sum_{i=1}^{2k} \min\{d_i(G),k\} - 2\sum_{i=2k+1}^n \max\{0,d_i(G)-k\}.
\]
So
\[
\eps_k(G)\le \frac{1}{2}\left(\sum_{i=1}^{2k} \min\{d_i(G),k\} - \sum_{i=2k+1}^n \max\{0,d_i(G)-k\}\right),
\]
as wanted.
\end{proof}

As a consequence of Theorem \ref{thm:main+bai}, we obtain the following result.

\begin{corollary}\label{cor:brouwer}
    Let $G=(V,E)$ be a graph with $|V|=n$. Then, for all $1\le k\le n-1$,
    \[
    \sum_{i=1}^k \lambda_i(L(G))\le |E|+\binom{k+1}{2}+\min\left\{\binom{n-k-1}{2},\binom{k}{2}\right\}.
    \]
\end{corollary}
\begin{proof}
    First, assume $k\le (n-1)/2$. Note that, in this case, $\min\left\{\binom{n-k-1}{2},\binom{k}{2}\right\}=\binom{k}{2}$. Then, by Theorem \ref{thm:main+bai}, we have
    \[
        \eps_k(G) \le \frac{1}{2}\left( \sum_{i=1}^{2k} \min\{d_i(G),k\} - \sum_{i=2k+1}^n \max\{0,d_i(G)-k\}\right) \le \frac{1}{2} \cdot 2k \cdot k = k^2 = \binom{k+1}{2}+\binom{k}{2}.
    \]
    Now, assume that $k > (n-1)/2$. Note that, in this case,  $  \min\left\{\binom{n-k-1}{2},\binom{k}{2}\right\}=\binom{n-k-1}{2}$. By Lemma \ref{lemma:complement_eps}, we have
    \[
        \eps_k(G) = \eps_{n-k-1}(\overline{G})+ nk -\binom{n}{2}.
    \]
    Since $n-k-1\le (n-1)/2$, we have, by the first case, $\eps_{n-k-1}(\overline{G}) \le (n-k-1)^2$. Therefore,
    \[
        \eps_k(G)\le   (n-k-1)^2+ nk-\binom{n}{2} = \binom{k+1}{2} + \binom{n-k-1}{2}.
    \]
\end{proof}

Theorem \ref{thm:weak_brouwer_improved} follows immediately from Corollary \ref{cor:brouwer}.

\begin{proof}[Proof of Theorem \ref{thm:weak_brouwer_improved}]
    Let $k\ge 1$, and let $G=(V,E)$ be a graph with $|V|\ge k$. 
    If $|V|=k$, the claim is trivial, as $\sum_{i=1}^{|V|} \lambda_i(L(G)) = 2|E|\le |E|+|V|^2$.
    So, we may assume $|V|\ge k+1$. Then, by Corollary \ref{cor:brouwer}, we obtain
    \[
    \sum_{i=1}^k \lambda_i(L(G))\le |E|+\binom{k+1}{2}+\min\left\{\binom{n-k-1}{2},\binom{k}{2}\right\} \le |E|+\binom{k+1}{2}+\binom{k}{2}= |E|+k^2.
    \]
\end{proof}

\section{Additional applications}\label{sec:apps}

In this section, we present some additional applications of our main result in the graphical setting. First, we present the following simple consequence of Theorem \ref{thm:degree_bound}. 
For a graph $G=(V,E)$ and a set $S\subset V$, let $G[S]$ be the subgraph of $G$ induced by $S$ (that is, the graph on vertex set $S$ with edge set $\{e\in E:\, e\subset S\}$). We denote the edge set of $G[S]$ by $E(G[S])$.

\begin{proposition}\label{prop:2k_set}
    Let $G=(V,E)$ be a graph with $|V|=n$, and let $1\le k\le n/2$. Then,
    \[
        \sum_{i=1}^k \lambda_i(L(G))\le |E| + \max \left\{ |E(G[S])| :\, S\subset V,\, |S|=2k\right\}.
    \]
\end{proposition}
\begin{proof}
    We order the vertices of $G$ as $v_1,\ldots,v_n$, where $\deg_G(v_i)=d_i(G)$ for all $1\le i\le n$. Let $S'=\{v_1,\ldots,v_{2k}\}$.
    Note that
    \[
        \sum_{i=1}^{2k} d_i(G) = \left| \{e\in E:\, |e\cap S'|=1\}\right| + 2|E(G[S'])| \le |E|+|E(G[S'])|.
    \]
 So, by Theorem \ref{thm:degree_bound}, we obtain
 \[
    \sum_{i=1}^k \lambda_i(L(G)) \le \sum_{i=1}^{2k} d_i(G) \le |E|+|E(G[S'])| \le  |E| + \max \left\{ |E(G[S])| :\, S\subset V,\, |S|=2k\right\}.
 \]
\end{proof}

We say that a family of graphs $\mathcal{G}$ is \emph{hereditary} if for every $G\in\mathcal{G}$, all induced subgraphs of $G$ belong to $\mathcal{G}$ as well.

\begin{lemma}\label{lemma:apps}
    Let $\mathcal{G}$ be a hereditary family of graphs. Let $f:\mathbb{Z}_{\ge 0} \to \mathbb{R}_{\ge 0}$ be a monotone non-decreasing function such that $|E|\le f(|V|)$ for every $G=(V,E)\in \mathcal{G}$. Then, for all $k\ge 1$ and $G=(V,E)\in \mathcal{G}$ with $|V|\ge k$,
    \[
        \sum_{i=1}^k \lambda_i(L(G))\le |E|+f(2k).
    \]
\end{lemma}
\begin{proof}
    Let $k\ge 1$ and $G=(V,E)\in\mathcal{G}$ with $|V|\ge k$. 
    If $k>|V|/2$, the claim follows trivially, since
    \[
        \sum_{i=1}^k \lambda_i(L(G)) \le 2|E|\le |E|+ f(|V|)\le |E|+ f(2k). 
    \]
    So, we may assume $k\le |V|/2$. Let $S\subset V$. Since $\mathcal{G}$ is hereditary, we have $G[S]\in \mathcal{G}$. Therefore, $|E(G[S])|\le f(|S|)$. So, by Proposition \ref{prop:2k_set}, we obtain
    \[
         \sum_{i=1}^k \lambda_i(L(G)) \le |E|+ \max \left\{ |E(G[S])| :\, S\subset V,\, |S|=2k\right\} \le  |E|+f(2k),
    \]
    as wanted.
\end{proof}

The next result collects several immediate applications of Lemma \ref{lemma:apps}, which may be of interest.

\begin{proposition}\label{prop:apps}
   Let $k\ge 1$, and let $G=(V,E)$ be a graph with $|V|\ge k$. Then, the following results hold.
   \begin{enumerate}
       \item (Haemers, Mohammadian, Tayfeh-Rezaie {\cite[Theorem 5]{haemers2010onthesum}}\footnote{In \cite{fritscher2011onthesum}, Fritscher, Hoppen, Rocha, and Trevisan
       proved a stronger bound for trees: $\eps_k(G)\le 2k-1 -(2k-2)/|V|$.}) If $G$ is a forest, then 
       \[
       \sum_{i=1}^k \lambda_i(L(G)) \le |E|+2k-1.
       \]

      \item (See Cooper \cite{cooper2021constraints}\footnote{In \cite{cooper2021constraints}, a weaker bound $\eps_k(G)\le (2k-1)(\lfloor\Delta/2\rfloor+1)$ was proved.}) If $G$ has maximum degree at most $\Delta$, then
       \[
         \sum_{i=1}^k \lambda_i(L(G))\le |E|+ k\cdot \min\{\Delta,k\}.
       \]
       
       \item (See Cooper \cite{cooper2021constraints}\footnote{In \cite{cooper2021constraints}, a slightly weaker bound $\eps_k(G)\le 6k-3$ was proved.}) If $G$ is planar, then
       \[
        \sum_{i=1}^k \lambda_i(L(G))\le |E|+ 6k-6.
       \]
       In particular, the bound in Conjecture \ref{conj:brouwer} holds for planar graphs for all $k\ge 10$.

       \item If $G$ is square-free (that is, it has no cycles of length $4$), then
       \[
       \sum_{i=1}^k \lambda_i(L(G))\le |E|+\left\lfloor k(1+\sqrt{8k-3})/2\right\rfloor.
       \]
       In particular, the bound in Conjecture \ref{conj:brouwer} holds for square-free graphs for all $k\ge 7$.
       \item If $G$ has girth at least $5$ (that is, it has no cycles of length smaller than $5$), then
       \[
       \sum_{i=1}^k \lambda_i(L(G))\le |E|+\left\lfloor k \sqrt{2k-1}\right\rfloor.
       \]
       In particular, the bound in Conjecture \ref{conj:brouwer} holds for graphs with girth at least $5$ for all $k\ge 1$.

       \item Let $t\ge 1$. If $G$ has no path of length $t$, then
       \[
            \sum_{i=1}^k \lambda_i(L(G))\le |E|+k(t-1). 
       \]
       Similarly, for $t\ge 2$, if $G$ has no cycles of length greater than $t$, then
       \[
        \sum_{i=1}^k \lambda_i(L(G))\le |E|+\left\lfloor t(2k-1)/2\right\rfloor.
       \]
   \end{enumerate}
\end{proposition}
\begin{proof}
    First, note that all the properties in the claim are hereditary, so we may apply Lemma \ref{lemma:apps} in each case.
    \begin{enumerate}
        \item Since every forest on $n$ vertices has at most $n-1$ edges, we obtain, by Lemma \ref{lemma:apps},
        \[
            \eps_k(G)\le 2k-1.
        \]
        \item If $\Delta\ge k$, the bound follows immediately from Theorem \ref{thm:weak_brouwer_improved}. So, we may assume $\Delta<k$. 
        An $n$-vertex graph with maximum degree at most $\Delta$ has at most $n\Delta/2$ edges. Therefore, by Lemma \ref{lemma:apps},
        \[
            \eps_k(G)\le 2k\cdot\Delta/2= k\Delta= k \cdot \min\{\Delta,k\}.
        \]
        \item Since every $n$-vertex planar graph has at most $3n-6$ edges (see, for example, \cite[Chapter I, Theorem 16]{bollobas1998modern}), we obtain, by Lemma \ref{lemma:apps},
        \[
            \eps_k(G)\le 3\cdot2k-6=6k-6.
        \]
        For $k\ge 10$, we have $6k-6\le \binom{k+1}{2}$, and therefore the bound in Conjecture \ref{conj:brouwer} holds for planar graphs for all $k\ge 10$.
        \item It is known (see, for example, \cite[Chapter IV, Theorem 12]{bollobas1998modern}) that a square-free graph on $n$ vertices has at most $\left\lfloor n(1+\sqrt{4n-3})/4\right\rfloor$ edges. Therefore, by Lemma \ref{lemma:apps},
        \[
            \eps_k(G)\le \left\lfloor2k(1+\sqrt{4\cdot 2k-3})/4 \right\rfloor=\left\lfloor k(1+\sqrt{8k-3})/2\right\rfloor.
        \]
        Since $\left\lfloor k(1+\sqrt{8k-3})/2\right\rfloor\le \binom{k+1}{2}$ for $k\ge 7$, we obtain that square-free graphs satisfy the bound in Conjecture \ref{conj:brouwer} for all $k\ge 7$.
        \item It was shown by Dutton and Brigham in \cite[Theorem 4]{dutton1991edges} (see also \cite[Theorem 4.2]{vanlint2001course}), that an $n$-vertex graph with girth at least $5$ has at most $\left\lfloor n\sqrt{n-1}/2\right\rfloor $ edges. Therefore, by Lemma \ref{lemma:apps},
        \[
            \eps_k(G)\le \left\lfloor 2k\sqrt{2k-1}/2\right\rfloor =\left\lfloor k\sqrt{2k-1}\right\rfloor.
        \]
        Since $\left\lfloor k\sqrt{2k-1}\right\rfloor \le \binom{k+1}{2}$ for $k\ge 1$, we obtain that graphs with girth at least $5$ satisfy the bound in Conjecture \ref{conj:brouwer} for all $k\ge 1$.

        \item  Since a graph on $n$ vertices with no path of length $t$ has at most $\left\lfloor (t-1)n/2\right\rfloor$ edges (see \cite[Chapter IV, Theorem 3]{bollobas1998modern}), we obtain, by Lemma \ref{lemma:apps},
        \[
        \eps_k(G)\le \left\lfloor (t-1)\cdot 2k/2\right\rfloor = (t-1)k.
        \]
        Similarly, since an $n$-vertex graph with no cycles of length greater than $t$ has at most $\left\lfloor t(n-1)/2\right\rfloor$ edges (see \cite[Chapter IV, Theorem 4]{bollobas1998modern}), we obtain, by Lemma \ref{lemma:apps},
        \[
        \eps_k(G)\le \left\lfloor t(2k-1)/2 \right\rfloor.
        \]
    \end{enumerate}
\end{proof}

\section{Laplacian eigenvalues of $(r+1)$-partite $r$-dimensional simplicial complexes}\label{sec:partite}

In this section, we prove Theorem \ref{thm:degree_bound_partite} and Corollary \ref{cor:dr_partite}.
Let $X$ be an $(r+1)$-partite $r$-dimensional simplicial complex with partite sets $V_1,\ldots,V_{r+1}$. Recall that for $1\le j\le r+1$, we denote by $X(r-1;j)$ the set of $(r-1)$-dimensional simplices of $X$ that do not intersect $V_j$, and for $1\le i\le |X(r-1;j)|$, we denote by 
$d^{(r)}_i(X;j)$ the $i$-th largest $r$-degree of a simplex in $X(r-1;j)$.

\begin{proof}[Proof of Theorem \ref{thm:degree_bound_partite}]
Let $X$ be an $(r+1)$-partite $r$-dimensional simplicial complex, and let $1\le k\le f_{r-1}(X)$.

By Corollary \ref{cor:Lminus_L}, $L_{r-1}^+(X)$ has at most $f_r(X)$ non-zero eigenvalues. Hence, if $f_r(X)<k$, we have
\[
   \sum_{i=1}^k \lambda_i(L_{r-1}^+(X)) = \sum_{i=1}^{f_{r-1}(X)}\lambda_i(L_{r-1}^+(X)) = \text{Trace}(L_{r-1}^{+}(X))= \sum_{\sigma\in X(r-1)}\deg_{X}^{(r)}(\sigma).
\]
On the other hand, since $f_r(X)<k$ and every $r$-dimensional simplex in $X$ contains exactly one $(r-1)$-dimensional face in $X(r-1;j)$, for all $1\le j\le r+1$, each $X(r-1;j)$ contains at most $k$ simplices with positive $r$-degree. Thus,
\[
     \sum_{i=1}^k \lambda_i(L_{r-1}^+(X)) = \sum_{\sigma\in X(r-1)}\deg_{X}^{(r)}(\sigma)=\sum_{j=1}^{r+1} \sum_{i=1}^{\min\{k,|X(r-1;j)|\}} d_i^{(r)}(X;j).
\]

So, we may assume $f_r(X)\ge k$. 
For $1\le j\le r+1$, let $L_j\in \Rea^{f_{r}(X)\times f_r(X)}$ be defined by
\begin{equation}\label{eq:Lj_def}
    (L_j)_{\tau,\eta}=\begin{cases}
    1 & \text{if } \tau=\eta, \\
    (\tau:\tau\cap \eta)(\eta:\tau\cap \eta) & \text{if } |\tau\cap \eta|=r,\, \tau\cap \eta\in X(r-1;j),\\
    0 & \text{otherwise,}
    \end{cases}
\end{equation}
for all $\tau,\eta\in X(r)$. By Lemma \ref{lemma:LAminus}, the multi-set of non-zero eigenvalues of $L_j$ consists of the non-zero elements in $\{\deg_X^{(r)}(\sigma):\, \sigma\in X(r-1;j)\}$. In particular,
\[
    \sum_{i=1}^{k} \lambda_i(L_j) = \sum_{i=1}^{\min\{k,|X(r-1;j)|\}} d_i^{(r)}(X;j).
\]
It is easy to check, by \eqref{eq:Lminus_formula} and \eqref{eq:Lj_def}, that $L^-_r(X)= \sum_{j=1}^{r+1}L_j$. Therefore, by Lemma \ref{lemma:ky_fan},
    \[
        \sum_{i=1}^k \lambda_i(L^+_{r-1}(X))\le
        \sum_{j=1}^{r+1} \sum_{i=1}^{k} \lambda_i(L_j)
        =\sum_{j=1}^{r+1} \sum_{i=1}^{\min\{k,|X(r-1;j)|\}} d_{i}^{(r)}(X;j).
    \]

\end{proof}

\begin{proof}[Proof of Corollary \ref{cor:dr_partite}]

    Let $X$ be an $(r+1)$-partite $r$-dimensional simplicial complex on vertex set $V$, with partite sets $V_1,\ldots,V_{r+1}$. 
    For $1\le j\le r+1$, let $A_j$ be the set consisting of the $\min\{k,|X(r-1;j)|\}$ simplices with largest $r$-degree in $X(r-1;j)$. Then,
    \begin{align*}
  \sum_{j=1}^{r+1} \sum_{i=1}^{\min\{k,|X(r-1;j)|\}} d_{i}^{(r)}(X;j) &= \sum_{j=1}^{r+1} \left|\left\{ (\sigma,\tau) :\, \sigma\in A_j,\, \tau\in X(r),\, \sigma\subset\tau\right\}\right|
  \\ &= \sum_{j=1}^{r+1} \left|\left\{ (v,\tau) : \, 
    v\in V_j,\, \tau\in X(r) ,\, v\in \tau,\, \tau\setminus\{v\}\in A_j \right\}\right|
   \\ &\le \sum_{j=1}^{r+1} \sum_{v\in V_j} \min\{\deg^{(r)}_X(v),|A_j|\}
   = \sum_{j=1}^{r+1} \sum_{v\in V_j} \min\{\deg^{(r)}_X(v),k\}
   \\&= \sum_{v\in V} \min\{\deg^{(r)}_X(v),k\} = \sum_{i=1}^k \left|\left\{ v\in V:\, \deg_X^{(r)}(v)\ge i\right\}\right|.
   \end{align*}
   Therefore, the claim follows from Theorem \ref{thm:degree_bound_partite}.
\end{proof}

\section{Concluding remarks}\label{sec:conclusion}

Motivated by Theorem \ref{thm:main+bai}, and based on some computer experiments on small graphs, we propose the following conjecture.
\begin{conjecture}\label{conj:brouwer_plus}
Let $G=(V,E)$ be a graph, and let $1\le k\le |V|$. Then,
\[
    \sum_{i=1}^k \lambda_i(L(G)) \le |E|+\frac{k}{2}+\frac{1}{2}\sum_{i=1}^k\min\{d_i(G),k\}.
\]
\end{conjecture}
Note that Conjecture \ref{conj:brouwer_plus} implies Brouwer's conjecture (Conjecture \ref{conj:brouwer}). 

\subsection{An extension of Brouwer's conjecture to simplicial complexes}\label{sec:conclusion:brower_for_complexes}

In Conjecture \ref{conj:higher_brouwer} we propose, motivated by the work of Abebe and Pfeffer in \cite{abebepartial,abebe2019conjectural}, that for every simplicial complex $X$, $1\le r\le \dim(X)$, and $1\le k\le f_{r-1}(X)$, the bound
\[
    \sum_{i=1}^k \lambda_i(L_{r-1}^{+}(X)) \le f_{r}(X) + \binom{k}{2}+rk
\] 
may hold. For $r=1$, this reduces to Brouwer's conjecture. This bound cannot be improved, as shown next.

First, we will need the following simple lemma. For a simplicial complex $X$ on vertex set $V$ and a set $\sigma$ disjoint from $V$, we denote by $X\ast \sigma$ the simplicial complex on vertex set $V\cup \sigma$ whose simplices are all the sets of the form $\tau\cup\eta$ for $\tau\in X$ and $\eta\subset \sigma$.

\begin{lemma}\label{lemma:coning}
  Let $X$ be an $r$-dimensional simplicial complex on vertex set $V$, and let $\sigma\ne\emptyset$ be a set disjoint from $V$. 
  Let $S$ be the multi-set consisting of all the eigenvalues of $L_{r}^-(X)$. Then, the non-zero eigenvalues of $L^{+}_{r-1+|\sigma|}(X\ast\sigma)$ are
  \[    
    \{\lambda+|\sigma| :\, \lambda\in S\}.
  \]
\end{lemma}
\begin{proof}
    Let $<$ be a linear order of $V$. We extend it to a linear order on $V\cup\sigma$ such that $v<w$ for all $v\in V$ and $w\in \sigma$. Note that there is a bijection $\tau\mapsto \tau\cup\sigma$ between the $r$-dimensional faces of $X$ and the $(r+|\sigma|)$-dimensional faces of $X\ast \sigma$. It is easy to check, using \eqref{eq:Lminus_formula}, that, under a choice of row and column order respecting this bijection, we have
    \[
        L^-_{r+|\sigma|}(X\ast\sigma)=|\sigma|I+L^-_{r}(X),
    \]
    where $I$ is the identity matrix. Therefore, the claim follows from Corollary \ref{cor:Lminus_L}.
\end{proof}

\begin{proposition}\label{prop:extremal_brouwer_complexes}
    Let $G=(V,E)$ be a graph, and let $1\le k\le \min\{|V|,|E|\}$. Let $\eps_k(G)= \sum_{i=1}^k \lambda_i(L(G)) - |E|$. 
    Let $r\ge 2$, and let $\sigma$ be a set of size $r-1$ disjoint from $V$. Let 
    $Y=G\ast\sigma$.
    Then,
    \[
    \sum_{i=1}^{k} \lambda_i(L_{r-1}^+(Y)) = f_r(Y) + \eps_k(G) + (r-1)k.
    \]
\end{proposition}
\begin{proof}
Note that $k\le |V|\;\le f_{r-1}(Y)$.
By Lemma \ref{lemma:coning}, the $k$ largest eigenvalues of $L_{r-1}^{+}(Y)$ are
\[
    \left\{\lambda_i(L_1^-(G))+r-1 :\, 1\le i\le k\right\}.
\]
So, using Corollary \ref{cor:Lminus_L}, we obtain
\[
\sum_{i=1}^k \lambda_i(L_{r-1}^+(Y)) = k(r-1) + \sum_{i=1}^k \lambda_i(L_1^{-}(G)) = k(r-1) + \sum_{i=1}^k \lambda_i(L(G)). 
\]
Since $f_r(Y)=|E|$, we obtain
\[
\sum_{i=1}^{k} \lambda_i(L_{r-1}^+(Y))= \left(\sum_{i=1}^k \lambda_i(L(G))\right) + k(r-1) = f_r(Y) + \eps_k(G) + (r-1)k.
\]
\end{proof}
By Proposition \ref{prop:extremal_brouwer_complexes}, every extremal example to the $k$-th inequality in Brouwer's conjecture (that is, every graph $G$ with $\eps_k(G)=\binom{k+1}{2}$) gives rise, for every $r\ge 2$, to a simplicial complex $Y$ achieving equality in the $k$-th inequality for $L_{r-1}^+(Y)$ in Conjecture \ref{conj:higher_brouwer}.
Such graphs are known to exist: for example, it is easy to check that the graph $G=(V,E)$ with $V=A\cup B$, $A\cap B=\emptyset$, $|A|=k$, and $E=\{\{u,v\}:\, |\{u,v\}\cap B|\le 1\}$, satisfies $\eps_k(G)=\binom{k+1}{2}$.

Finally, the next proposition shows that the $k=1$ case of Conjecture \ref{conj:higher_brouwer} follows from a result of Fan, Wu, and Wang \cite{fan2025largest}.

\begin{proposition}
   Let $X$ be a simplicial complex on vertex set $V$, and let $1\le r\le \dim(X)$. Then,
\[
    \lambda_1(L_{r-1}^{+}(X)) \le f_{r}(X) +r.
\] 
\end{proposition}
\begin{proof}
    For $\tau\in X(r)$, let $\partial\tau=\{\sigma\in X(r-1):\, \sigma\subset\tau\}$. 
    It is shown in \cite[Theorem 3.5]{fan2025largest} that
    \[
        \lambda_1(L_{r-1}^+(X))\le \max_{\tau\in X(r)} \left(\sum_{\sigma\in \partial \tau} \deg_X^{(r)}(\sigma)\right).
    \]
    Now, let $\tau\in X(r)$. For every $\sigma,\sigma'\in \partial\tau$, $\tau$ is the unique $r$-dimensional simplex containing both $\sigma$ and $\sigma'$. Therefore,
    \[
   \sum_{\sigma\in \partial \tau} \deg_X^{(r)}(\sigma)\le  f_r(X)+ |\partial \tau|-1 = f_r(X)+r.
    \]
    So, $ \lambda_1(L_{r-1}^+(X))\le f_r(X)+r$, as required.
\end{proof}

\subsection{Sums of eigenvalues of signless Laplacian matrices}

Let $X$ be a simplicial complex, and let $1\le r\le \dim(X)$. The $(r-1)$-dimensional upper \emph{signless Laplacian} on $X$ is the matrix $Q_{r-1}^{+}(X)\in \Rea^{f_{r-1}(X)\times f_{r-1}(X)}$ defined by
\[
    Q_{r-1}^{+}(X)_{\tau,\eta}= \begin{cases}
        \deg_X^{(r)}(\tau) & \text{if } \tau=\eta,\\
        1 & \text{if } |\tau\cap \eta|=r-1,\, \tau\cup\eta\in X,\\
        0 & \text{otherwise,}
    \end{cases}
\]
for all $\tau,\eta\in X(r-1)$. For a graph $G$,  $Q_0^{+}(G)=Q(G)$ is the well-known signless Laplacian matrix of $G$. 
The \emph{$r$-dimensional signless boundary operator} $N_r(X)\in \Rea^{f_{r-1}(X)\times f_r(X)}$ is defined by
\[
    N_r(X)_{\sigma,\tau} = \begin{cases}
        1 & \text{if } \sigma\subset\tau,\\
        0 & \text{otherwise,}
    \end{cases}
\]
for all $\sigma\in X(r-1)$ and $\tau\in X(r)$. Note that $Q^{+}_{r-1}(X)= N_r(X) N_r(X)^{T}$. We define $Q^{-}_r(X)= N_r(X)^T N_r(X)\in \Rea^{f_r(X)\times f_r(X)}$. By Lemma \ref{lemma:AAT}, $Q_{r-1}^{+}(X)$ and $Q_r^{-}(X)$ have the same non-zero eigenvalues.

In direct analogy to Brouwer's conjecture, Ashraf, Omidi, and Tayfeh-Rezaie conjectured in \cite{ashraf2013onthesum} that for every $k\ge 1$ and every graph $G=(V,E)$ with $|V|\ge k$, $\sum_{i=1}^k \lambda_i(Q(G))\le |E|+\binom{k+1}{2}$. They verified their conjecture for $k=1$ and $k=2$, and for general $k$ in the special case of regular graphs. See, for example, \cite{yang2014onaconjecture,chen2018note2,pirzada2024kyfan, zhou2024onthesum} and the references therein for some further partial results.

Let us note that the signless analogue of our main result, Theorem \ref{thm:degree_bound}, holds.

\begin{theorem}\label{thm:degree_bound_signless}
      Let $X$ be a simplicial complex on vertex set $V$, $1\le r\le \dim(X)$, and $1\le k\le f_{r-1}(X)/(r+1)$. Then,
    \[
        \sum_{i=1}^k \lambda_i(Q_{r-1}^{+}(X))\le \sum_{i=1}^{(r+1)k} d^{(r)}_i(X).
    \]
\end{theorem}
Since the proof of Theorem \ref{thm:degree_bound_signless} is essentially the same as the proof of Theorem \ref{thm:degree_bound}, we omit the details. Signless Laplacian versions of Corollary \ref{cor:weak_brouwer_complexes}, Theorem \ref{thm:degree_bound_partite}, Corollary \ref{cor:dr_partite},
Proposition \ref{prop:2k_set}, Lemma \ref{lemma:apps} and parts $1,3,4,5$ and $6$ of Proposition \ref{prop:apps} hold as well. Again, their proofs are virtually the same as for their Laplacian counterparts.

On the other hand, since an analogue of Bai's theorem does not hold for the signless Laplacian, our proof of Theorem \ref{thm:weak_brouwer_improved} does not extend to the signless setting. However, we can prove the following weaker result.
\begin{proposition}
    Let $k\ge 1$, and let $G=(V,E)$ be a triangle-free graph with $|V|\ge k$. Then,
    \[
        \sum_{i=1}^k \lambda_i(Q(G))\le |E|+k^2.
    \]
\end{proposition}
\begin{proof}
    Let $G=(V,E)$ be a triangle-free graph with $|V|\ge k$. 
    Since the family of triangle-free graphs is hereditary and, by Mantel's theorem (see \cite[Chapter I, Theorem 2]{bollobas1998modern}),  every $n$-vertex triangle-free graph has at most $n^2/4$ edges, we obtain, by the signless Laplacian version of Lemma \ref{lemma:apps},
    \[
         \sum_{i=1}^k \lambda_i(Q(G))\le |E|+ (2k)^2/4= |E|+k^2.
    \]
\end{proof}

Computer experiments suggest that analogues of Theorem \ref{thm:main+bai} and Conjecture \ref{conj:brouwer_plus} may hold in the signless setting as well. It is plausible that a signless version of Conjecture \ref{conj:higher_brouwer} may hold as well.

\bibliographystyle{abbrv}
\bibliography{main}

\begin{bibdiv}
\begin{biblist}

\bib{abebe2019conjectural}{article}{
      author={Abebe, Rediet},
       title={A conjectural {B}rouwer inequality for higher-dimensional {L}aplacian spectra},
        date={2019},
     journal={arXiv preprint arXiv:1907.07541},
}

\bib{abebepartial}{article}{
      author={Abebe, Rediet},
      author={Pfeffer, Joshua},
       title={On the partial sum of the {L}aplacian eigenvalues of abstract simplicial complexes},
        note={Unpublished report, \url{https://www-users.cse.umn.edu/~reiner/REU/AbebePfeffer2012.pdf}},
}

\bib{aharoni2005eigenvalues}{article}{
      author={Aharoni, R.},
      author={Berger, E.},
      author={Meshulam, R.},
       title={Eigenvalues and homology of flag complexes and vector representations of graphs},
        date={2005},
        ISSN={1016-443X,1420-8970},
     journal={Geom. Funct. Anal.},
      volume={15},
      number={3},
       pages={555\ndash 566},
         url={https://doi.org/10.1007/s00039-005-0516-9},
      review={\MR{2221142}},
}

\bib{alon1985lambda1}{article}{
      author={Alon, N.},
      author={Milman, V.~D.},
       title={{$\lambda_1,$} isoperimetric inequalities for graphs, and superconcentrators},
        date={1985},
        ISSN={0095-8956,1096-0902},
     journal={J. Combin. Theory Ser. B},
      volume={38},
      number={1},
       pages={73\ndash 88},
         url={https://doi.org/10.1016/0095-8956(85)90092-9},
      review={\MR{782626}},
}

\bib{anari2024log}{article}{
      author={Anari, Nima},
      author={Liu, Kuikui},
      author={Oveis~Gharan, Shayan},
      author={Vinzant, Cynthia},
       title={Log-concave polynomials {II}: {H}igh-dimensional walks and an {FPRAS} for counting bases of a matroid},
        date={2024},
        ISSN={0003-486X,1939-8980},
     journal={Ann. of Math. (2)},
      volume={199},
      number={1},
       pages={259\ndash 299},
         url={https://doi.org/10.4007/annals.2024.199.1.4},
      review={\MR{4681146}},
}

\bib{anderson1985eigenvalues}{article}{
      author={Anderson, William~N., Jr.},
      author={Morley, Thomas~D.},
       title={Eigenvalues of the {L}aplacian of a graph},
        date={1985},
        ISSN={0308-1087,1563-5139},
     journal={Linear and Multilinear Algebra},
      volume={18},
      number={2},
       pages={141\ndash 145},
         url={https://doi.org/10.1080/03081088508817681},
      review={\MR{817657}},
}

\bib{ashraf2013onthesum}{article}{
      author={Ashraf, F.},
      author={Omidi, G.~R.},
      author={Tayfeh-Rezaie, B.},
       title={On the sum of signless {L}aplacian eigenvalues of a graph},
        date={2013},
        ISSN={0024-3795,1873-1856},
     journal={Linear Algebra Appl.},
      volume={438},
      number={11},
       pages={4539\ndash 4546},
         url={https://doi.org/10.1016/j.laa.2013.01.023},
      review={\MR{3034549}},
}

\bib{bai2011gronemerris}{article}{
      author={Bai, Hua},
       title={The {G}rone-{M}erris conjecture},
        date={2011},
        ISSN={0002-9947,1088-6850},
     journal={Trans. Amer. Math. Soc.},
      volume={363},
      number={8},
       pages={4463\ndash 4474},
         url={https://doi.org/10.1090/S0002-9947-2011-05393-6},
      review={\MR{2792996}},
}

\bib{ballmann1997l2}{article}{
      author={Ballmann, W.},
      author={{\'S}wiatkowski, J.},
       title={On {$L^2$}-cohomology and property ({T}) for automorphism groups of polyhedral cell complexes},
        date={1997},
        ISSN={1016-443X,1420-8970},
     journal={Geom. Funct. Anal.},
      volume={7},
      number={4},
       pages={615\ndash 645},
         url={https://doi.org/10.1007/s000390050022},
      review={\MR{1465598}},
}

\bib{berndsen2010thesis}{thesis}{
      author={Berndsen, Jochem},
       title={Three problems in algebraic combinatorics},
        type={Master's Thesis},
organization={Eindhoven University of Technology},
        date={2010},
}

\bib{bollobas1998modern}{book}{
      author={Bollob\'as, B\'ela},
       title={Modern graph theory},
      series={Graduate Texts in Mathematics},
   publisher={Springer-Verlag, New York},
        date={1998},
      volume={184},
        ISBN={0-387-98488-7},
         url={https://doi.org/10.1007/978-1-4612-0619-4},
      review={\MR{1633290}},
}

\bib{brouwer2008lower}{article}{
      author={Brouwer, Andries~E.},
      author={Haemers, Willem~H.},
       title={A lower bound for the {L}aplacian eigenvalues of a graph---proof of a conjecture by {G}uo},
        date={2008},
        ISSN={0024-3795,1873-1856},
     journal={Linear Algebra Appl.},
      volume={429},
      number={8-9},
       pages={2131\ndash 2135},
         url={https://doi.org/10.1016/j.laa.2008.06.008},
      review={\MR{2446646}},
}

\bib{brouwer2012book}{book}{
      author={Brouwer, Andries~E.},
      author={Haemers, Willem~H.},
       title={Spectra of graphs},
      series={Universitext},
   publisher={Springer, New York},
        date={2012},
        ISBN={978-1-4614-1938-9},
         url={https://doi.org/10.1007/978-1-4614-1939-6},
      review={\MR{2882891}},
}

\bib{caputo2010aldous}{article}{
      author={Caputo, Pietro},
      author={Liggett, Thomas~M.},
      author={Richthammer, Thomas},
       title={Proof of {A}ldous' spectral gap conjecture},
        date={2010},
        ISSN={0894-0347,1088-6834},
     journal={J. Amer. Math. Soc.},
      volume={23},
      number={3},
       pages={831\ndash 851},
         url={https://doi.org/10.1090/S0894-0347-10-00659-4},
      review={\MR{2629990}},
}

\bib{chen2019onbrouwers}{article}{
      author={Chen, Xiaodan},
       title={On {B}rouwer's conjecture for the sum of {$k$} largest {L}aplacian eigenvalues of graphs},
        date={2019},
        ISSN={0024-3795,1873-1856},
     journal={Linear Algebra Appl.},
      volume={578},
       pages={402\ndash 410},
         url={https://doi.org/10.1016/j.laa.2019.05.029},
      review={\MR{3956797}},
}

\bib{chen2018note2}{article}{
      author={Chen, Xiaodan},
      author={Hao, Guoliang},
      author={Jin, Dequan},
      author={Li, Jingjian},
       title={Note on a conjecture for the sum of signless {L}aplacian eigenvalues},
        date={2018},
        ISSN={0011-4642,1572-9141},
     journal={Czechoslovak Math. J.},
      volume={68(143)},
      number={3},
       pages={601\ndash 610},
         url={https://doi.org/10.21136/CMJ.2018.0548-16},
      review={\MR{3851878}},
}

\bib{chen2018note}{article}{
      author={Chen, Xiaodan},
      author={Li, Jingjian},
      author={Fan, Yingmei},
       title={Note on an upper bound for sum of the {L}aplacian eigenvalues of a graph},
        date={2018},
        ISSN={0024-3795,1873-1856},
     journal={Linear Algebra Appl.},
      volume={541},
       pages={258\ndash 265},
         url={https://doi.org/10.1016/j.laa.2017.12.006},
      review={\MR{3742621}},
}

\bib{chen2025more}{article}{
      author={Chen, Xiaodan},
      author={Zi, Junwei},
       title={More on the full {B}rouwer {L}aplacian spectrum conjecture},
        date={2025},
     journal={arXiv preprint 2503.11165},
}

\bib{cooper2021constraints}{article}{
      author={Cooper, Joshua~N.},
       title={Constraints on {B}rouwer's {L}aplacian spectrum conjecture},
        date={2021},
        ISSN={0024-3795,1873-1856},
     journal={Linear Algebra Appl.},
      volume={615},
       pages={11\ndash 27},
         url={https://doi.org/10.1016/j.laa.2020.12.028},
      review={\MR{4198697}},
}

\bib{du2012upper}{article}{
      author={Du, Zhibin},
      author={Zhou, Bo},
       title={Upper bounds for the sum of {L}aplacian eigenvalues of graphs},
        date={2012},
        ISSN={0024-3795,1873-1856},
     journal={Linear Algebra Appl.},
      volume={436},
      number={9},
       pages={3672\ndash 3683},
         url={https://doi.org/10.1016/j.laa.2012.01.007},
      review={\MR{2900744}},
}

\bib{dutton1991edges}{article}{
      author={Dutton, R.~D.},
      author={Brigham, R.~C.},
       title={Edges in graphs with large girth},
        date={1991},
        ISSN={0911-0119,1435-5914},
     journal={Graphs Combin.},
      volume={7},
      number={4},
       pages={315\ndash 321},
         url={https://doi.org/10.1007/BF01787638},
      review={\MR{1143180}},
}

\bib{duval2002shifted}{article}{
      author={Duval, Art~M.},
      author={Reiner, Victor},
       title={Shifted simplicial complexes are {L}aplacian integral},
        date={2002},
        ISSN={0002-9947,1088-6850},
     journal={Trans. Amer. Math. Soc.},
      volume={354},
      number={11},
       pages={4313\ndash 4344},
         url={https://doi.org/10.1090/S0002-9947-02-03082-9},
      review={\MR{1926878}},
}

\bib{eckmann1945haromic}{article}{
      author={Eckmann, Beno},
       title={Harmonische {F}unktionen und {R}andwertaufgaben in einem {K}omplex},
        date={1945},
        ISSN={0010-2571,1420-8946},
     journal={Comment. Math. Helv.},
      volume={17},
       pages={240\ndash 255},
         url={https://doi.org/10.1007/BF02566245},
      review={\MR{13318}},
}

\bib{fan2025largest}{article}{
      author={Fan, Yi-Zheng},
      author={Wu, Hui-Feng},
      author={Wang, Yi},
       title={The largest {L}aplacian eigenvalue and the balancedness of simplicial complexes},
        date={2025},
        ISSN={0925-9899,1572-9192},
     journal={J. Algebraic Combin.},
      volume={61},
      number={4},
       pages={53},
         url={https://doi.org/10.1007/s10801-025-01419-1},
      review={\MR{4925290}},
}

\bib{fiedler1973algebraic}{article}{
      author={Fiedler, Miroslav},
       title={Algebraic connectivity of graphs},
        date={1973},
        ISSN={0011-4642},
     journal={Czechoslovak Math. J.},
      volume={23(98)},
       pages={298\ndash 305},
      review={\MR{318007}},
}

\bib{fritscher2011onthesum}{article}{
      author={Fritscher, Eliseu},
      author={Hoppen, Carlos},
      author={Rocha, Israel},
      author={Trevisan, Vilmar},
       title={On the sum of the {L}aplacian eigenvalues of a tree},
        date={2011},
        ISSN={0024-3795,1873-1856},
     journal={Linear Algebra Appl.},
      volume={435},
      number={2},
       pages={371\ndash 399},
         url={https://doi.org/10.1016/j.laa.2011.01.036},
      review={\MR{2782788}},
}

\bib{ganie2020further}{article}{
      author={Ganie, Hilal~A.},
      author={Pirzada, S.},
      author={Rather, Bilal~A.},
      author={Trevisan, Vilmar},
       title={Further developments on {B}rouwer's conjecture for the sum of {L}aplacian eigenvalues of graphs},
        date={2020},
        ISSN={0024-3795,1873-1856},
     journal={Linear Algebra Appl.},
      volume={588},
       pages={1\ndash 18},
         url={https://doi.org/10.1016/j.laa.2019.11.020},
      review={\MR{4037607}},
}

\bib{garland1973padic}{article}{
      author={Garland, Howard},
       title={{$p$}-adic curvature and the cohomology of discrete subgroups of {$p$}-adic groups},
        date={1973},
        ISSN={0003-486X},
     journal={Ann. of Math. (2)},
      volume={97},
       pages={375\ndash 423},
         url={https://doi.org/10.2307/1970829},
      review={\MR{320180}},
}

\bib{grone1994laplacian}{article}{
      author={Grone, Robert},
      author={Merris, Russell},
       title={The {L}aplacian spectrum of a graph {II}},
        date={1994},
        ISSN={0895-4801},
     journal={SIAM J. Discrete Math.},
      volume={7},
      number={2},
       pages={221\ndash 229},
         url={https://doi.org/10.1137/S0895480191222653},
      review={\MR{1271994}},
}

\bib{grone1995eigenvalues}{article}{
      author={Grone, Robert~D.},
       title={Eigenvalues and the degree sequences of graphs},
        date={1995},
        ISSN={0308-1087,1563-5139},
     journal={Linear and Multilinear Algebra},
      volume={39},
      number={1-2},
       pages={133\ndash 136},
         url={https://doi.org/10.1080/03081089508818384},
      review={\MR{1374475}},
}

\bib{guan2014onthesum}{article}{
      author={Guan, Mei},
      author={Zhai, Mingqing},
      author={Wu, Yongfeng},
       title={On the sum of the two largest {L}aplacian eigenvalues of trees},
        date={2014},
     journal={Journal of Inequalities and Applications},
      volume={2014},
      number={1},
       pages={242},
}

\bib{haemers2010onthesum}{article}{
      author={Haemers, W.~H.},
      author={Mohammadian, A.},
      author={Tayfeh-Rezaie, B.},
       title={On the sum of {L}aplacian eigenvalues of graphs},
        date={2010},
        ISSN={0024-3795,1873-1856},
     journal={Linear Algebra Appl.},
      volume={432},
      number={9},
       pages={2214\ndash 2221},
         url={https://doi.org/10.1016/j.laa.2009.03.038},
      review={\MR{2599854}},
}

\bib{hammer1996laplacian}{article}{
      author={Hammer, P.~L.},
      author={Kelmans, A.~K.},
       title={Laplacian spectra and spanning trees of threshold graphs},
        date={1996},
        ISSN={0166-218X,1872-6771},
     journal={Discrete Appl. Math.},
      volume={65},
      number={1-3},
       pages={255\ndash 273},
         url={https://doi.org/10.1016/0166-218X(94)00049-J},
      review={\MR{1380078}},
}

\bib{helmberg2017spectral}{article}{
      author={Helmberg, Christoph},
      author={Trevisan, Vilmar},
       title={Spectral threshold dominance, {B}rouwer's conjecture and maximality of {L}aplacian energy},
        date={2017},
        ISSN={0024-3795,1873-1856},
     journal={Linear Algebra Appl.},
      volume={512},
       pages={18\ndash 31},
         url={https://doi.org/10.1016/j.laa.2016.09.029},
      review={\MR{3567511}},
}

\bib{horak2013spectra}{article}{
      author={Horak, Danijela},
      author={Jost, J\"{u}rgen},
       title={Spectra of combinatorial {L}aplace operators on simplicial complexes},
        date={2013},
        ISSN={0001-8708,1090-2082},
     journal={Adv. Math.},
      volume={244},
       pages={303\ndash 336},
         url={https://doi.org/10.1016/j.aim.2013.05.007},
      review={\MR{3077874}},
}

\bib{horn2013matrix}{book}{
      author={Horn, Roger~A.},
      author={Johnson, Charles~R.},
       title={Matrix analysis},
     edition={Second},
   publisher={Cambridge University Press, Cambridge},
        date={2013},
        ISBN={978-0-521-54823-6},
      review={\MR{2978290}},
}

\bib{kahle2014sharp}{article}{
      author={Kahle, Matthew},
       title={Sharp vanishing thresholds for cohomology of random flag complexes},
        date={2014},
        ISSN={0003-486X,1939-8980},
     journal={Ann. of Math. (2)},
      volume={179},
      number={3},
       pages={1085\ndash 1107},
         url={https://doi.org/10.4007/annals.2014.179.3.5},
      review={\MR{3171759}},
}

\bib{kalai1983enumeration}{article}{
      author={Kalai, Gil},
       title={Enumeration of {${\bf Q}$}-acyclic simplicial complexes},
        date={1983},
        ISSN={0021-2172},
     journal={Israel J. Math.},
      volume={45},
      number={4},
       pages={337\ndash 351},
         url={https://doi.org/10.1007/BF02804017},
      review={\MR{720308}},
}

\bib{katz2005grone}{article}{
      author={Katz, Nets~Hawk},
       title={The {G}rone {M}erris conjecture and a quadratic eigenvalue problem},
        date={2005},
     journal={arXiv preprint math\slash0512647},
}

\bib{kirchhoff1847ueber}{article}{
      author={Kirchhoff, Gustav},
       title={Ueber die {A}ufl{\"o}sung der {G}leichungen, auf welche man bei der {U}ntersuchung der linearen {V}ertheilung galvanischer {S}tr{\"o}me gef{\"u}hrt wird},
        date={1847},
     journal={Annalen der Physik},
      volume={148},
      number={12},
       pages={497\ndash 508},
}

\bib{kook2000combinatorial}{article}{
      author={Kook, W.},
      author={Reiner, V.},
      author={Stanton, D.},
       title={Combinatorial {L}aplacians of matroid complexes},
        date={2000},
        ISSN={0894-0347,1088-6834},
     journal={J. Amer. Math. Soc.},
      volume={13},
      number={1},
       pages={129\ndash 148},
         url={https://doi.org/10.1090/S0894-0347-99-00316-1},
      review={\MR{1697094}},
}

\bib{lew2024partition}{article}{
      author={Lew, Alan},
       title={Partition density, star arboricity, and sums of {L}aplacian eigenvalues of graphs},
        date={2024},
     journal={arXiv preprint 2410.04563},
}

\bib{li2000note}{article}{
      author={Li, Jiong-Sheng},
      author={Pan, Yong-Liang},
       title={A note on the second largest eigenvalue of the {L}aplacian matrix of a graph},
        date={2000},
        ISSN={0308-1087,1563-5139},
     journal={Linear and Multilinear Algebra},
      volume={48},
      number={2},
       pages={117\ndash 121},
         url={https://doi.org/10.1080/03081080008818663},
      review={\MR{1813439}},
}

\bib{li1997upper}{article}{
      author={Li, Jiong-Sheng},
      author={Zhang, Xiao-Dong},
       title={A new upper bound for eigenvalues of the {L}aplacian matrix of a graph},
        date={1997},
        ISSN={0024-3795,1873-1856},
     journal={Linear Algebra Appl.},
      volume={265},
       pages={93\ndash 100},
         url={https://doi.org/10.1016/S0024-3795(96)00592-7},
      review={\MR{1466892}},
}

\bib{li2022full}{article}{
      author={Li, Wen-Jun},
      author={Guo, Ji-Ming},
       title={On the full {B}rouwer's {L}aplacian spectrum conjecture},
        date={2022},
        ISSN={0012-365X,1872-681X},
     journal={Discrete Math.},
      volume={345},
      number={12},
       pages={Paper No. 113078, 14},
         url={https://doi.org/10.1016/j.disc.2022.113078},
      review={\MR{4455926}},
}

\bib{marshall2011inequalities}{book}{
      author={Marshall, Albert~W.},
      author={Olkin, Ingram},
      author={Arnold, Barry~C.},
       title={Inequalities: theory of majorization and its applications},
     edition={Second edition},
      series={Springer Series in Statistics},
   publisher={Springer, New York},
        date={2011},
        ISBN={978-0-387-40087-7},
         url={https://doi.org/10.1007/978-0-387-68276-1},
      review={\MR{2759813}},
}

\bib{mayard2010thesis}{thesis}{
      author={Mayank},
       title={On variants of the {G}rone-{M}erris conjecture},
        type={Master's Thesis},
organization={Eindhoven University of Technology},
        date={2010},
}

\bib{merris1994survey}{article}{
      author={Merris, Russell},
       title={Laplacian matrices of graphs: a survey},
        date={1994},
        ISSN={0024-3795,1873-1856},
     journal={Linear Algebra Appl.},
      volume={197/198},
       pages={143\ndash 176},
         url={https://doi.org/10.1016/0024-3795(94)90486-3},
      review={\MR{1275613}},
}

\bib{mohar1989isoperimetric}{article}{
      author={Mohar, Bojan},
       title={Isoperimetric numbers of graphs},
        date={1989},
        ISSN={0095-8956,1096-0902},
     journal={J. Combin. Theory Ser. B},
      volume={47},
      number={3},
       pages={274\ndash 291},
         url={https://doi.org/10.1016/0095-8956(89)90029-4},
      review={\MR{1026065}},
}

\bib{oppenheim2018local}{article}{
      author={Oppenheim, Izhar},
       title={Local spectral expansion approach to high dimensional expanders {P}art {I}: {D}escent of spectral gaps},
        date={2018},
        ISSN={0179-5376,1432-0444},
     journal={Discrete Comput. Geom.},
      volume={59},
      number={2},
       pages={293\ndash 330},
         url={https://doi.org/10.1007/s00454-017-9948-x},
      review={\MR{3755725}},
}

\bib{pirzada2024kyfan}{article}{
      author={Pirzada, S.},
      author={Ul~Shaban, Rezwan},
      author={Ganie, Hilal~A.},
      author={de~Lima, L.},
       title={On the {K}y {F}an norm of the signless {L}aplacian matrix of a graph},
        date={2024},
        ISSN={2238-3603,1807-0302},
     journal={Comput. Appl. Math.},
      volume={43},
      number={1},
       pages={Paper No. 26, 15},
         url={https://doi.org/10.1007/s40314-023-02561-x},
      review={\MR{4682809}},
}

\bib{rocha2020aas}{article}{
      author={Rocha, Israel},
       title={Brouwer's conjecture holds asymptotically almost surely},
        date={2020},
        ISSN={0024-3795,1873-1856},
     journal={Linear Algebra Appl.},
      volume={597},
       pages={198\ndash 205},
         url={https://doi.org/10.1016/j.laa.2020.03.019},
      review={\MR{4082064}},
}

\bib{torres2024brouwer}{article}{
      author={Torres, Guilherme~Simon},
      author={Trevisan, Vilmar},
       title={Brouwer's {C}onjecture for the {C}artesian product of graphs},
        date={2024},
        ISSN={0024-3795,1873-1856},
     journal={Linear Algebra Appl.},
      volume={685},
       pages={66\ndash 76},
         url={https://doi.org/10.1016/j.laa.2023.12.019},
      review={\MR{4686817}},
}

\bib{vanlint2001course}{book}{
      author={van Lint, J.~H.},
      author={Wilson, R.~M.},
       title={A course in combinatorics},
     edition={Second},
   publisher={Cambridge University Press, Cambridge},
        date={2001},
        ISBN={0-521-00601-5},
         url={https://doi.org/10.1017/CBO9780511987045},
      review={\MR{1871828}},
}

\bib{yang2014onaconjecture}{article}{
      author={Yang, Jieshan},
      author={You, Lihua},
       title={On a conjecture for the signless {L}aplacian eigenvalues},
        date={2014},
        ISSN={0024-3795,1873-1856},
     journal={Linear Algebra Appl.},
      volume={446},
       pages={115\ndash 132},
         url={https://doi.org/10.1016/j.laa.2013.12.032},
      review={\MR{3163132}},
}

\bib{zhou2024onthesum}{article}{
      author={Zhou, Zi-Ming},
      author={He, Chang-Xiang},
      author={Shan, Hai-Ying},
       title={On the sum of the first two largest signless {L}aplacian eigenvalues of a graph},
        date={2024},
        ISSN={0012-365X,1872-681X},
     journal={Discrete Math.},
      volume={347},
      number={9},
       pages={Paper No. 114035, 14},
         url={https://doi.org/10.1016/j.disc.2024.114035},
      review={\MR{4741007}},
}

\bib{zuk1996propriete}{article}{
      author={\.{Z}uk, Andrzej},
       title={La propri\'et\'e{} ({T}) de {K}azhdan pour les groupes agissant sur les poly\`edres},
        date={1996},
        ISSN={0764-4442},
     journal={C. R. Acad. Sci. Paris S\'er. I Math.},
      volume={323},
      number={5},
       pages={453\ndash 458},
      review={\MR{1408975}},
}

\end{biblist}
\end{bibdiv}

\end{document}